\documentclass[11pt]{article}
\def\thetitle{The diameter of uniform spanning trees in high dimensions}

\usepackage{amsmath,amssymb}

\usepackage[usenames,dvipsnames,svgnames,table]{xcolor}
\definecolor{CombinatoricaAqua}{HTML}{00698C}
\definecolor{CombinatoricaBlue}{HTML}{3A3293}
\definecolor{CombinatoricaBrown}{HTML}{66220C}
\definecolor{CombinatoricaRed}{HTML}{DF2A27}
\definecolor{HarvardCrimson}{rgb}{0.6471, 0.1098, 0.1882}

\makeatletter
\let\reftagform@=\tagform@	
\def\tagform@#1{\maketag@@@
	{(\ignorespaces\textcolor{CombinatoricaBrown}{#1}\unskip\@@italiccorr)}}
\renewcommand{\eqref}[1]{\textup{\reftagform@{\ref{#1}}}}
\makeatother

\usepackage[backref=page]{hyperref}
\hypersetup{
	unicode,
	pdfencoding=auto,
	pdfauthor={Peleg Michaeli, Asaf Nachmias and Matan Shalev},
	pdftitle={\thetitle},
	pdfsubject={},
	pdfkeywords={},
	colorlinks=true,
	citecolor=CombinatoricaBlue,
	linkcolor=CombinatoricaAqua,
	anchorcolor=CombinatoricaBrown,
	urlcolor=HarvardCrimson}

\usepackage{amsthm}
\usepackage{thmtools}
\usepackage{bbm}
\usepackage{enumitem}
\usepackage[capitalize]{cleveref}
\Crefname{fact}{Fact}{Facts}
\Crefname{claim}{Claim}{Claims}


\declaretheoremstyle[
spaceabove=\topsep, spacebelow=\topsep,
headfont=\color{CombinatoricaBrown}\normalfont\bfseries,
bodyfont=\itshape,
]{thm}
\declaretheoremstyle[
spaceabove=\topsep, spacebelow=\topsep,
headfont=\color{CombinatoricaBrown}\normalfont\bfseries,
bodyfont=\normalfont,
]{dfn}
\declaretheoremstyle[
spaceabove=0.5\topsep, spacebelow=0.5\topsep,
headfont=\color{CombinatoricaBrown}\normalfont\bfseries,
bodyfont=\normalfont,
]{rmk}

\declaretheorem[style=thm,parent=section]{theorem}
\declaretheorem[style=thm,sibling=theorem]{lemma}
\declaretheorem[style=thm,sibling=theorem]{corollary}
\declaretheorem[style=thm,sibling=theorem]{claim}

\declaretheorem[style=thm,sibling=theorem]{observation}

\declaretheorem[style=definition,numbered=no]{acknowledgements}

\usepackage[nobysame,msc-links]{amsrefs}

\BibSpec{book}{%
	+{}  {\PrintPrimary}                {transition}
	+{,} { \textbf}                     {title} 
	+{.} { }                            {part}
	+{:} { \textit}                     {subtitle}
	+{,} { \PrintEdition}               {edition}
	+{}  { \PrintEditorsB}              {editor}
	+{,} { \PrintTranslatorsC}          {translator}
	+{,} { \PrintContributions}         {contribution}
	+{,} { }                            {series}
	+{,} { \voltext}                    {volume}
	+{,} { }                            {publisher}
	+{,} { }                            {organization}
	+{,} { }                            {address}
	+{,} { \PrintDateB}                 {date}
	+{,} { }                            {status}
	+{}  { \parenthesize}               {language}
	+{}  { \PrintTranslation}           {translation}
	+{;} { \PrintReprint}               {reprint}
	+{.} { }                            {note}
	+{.} {}                             {transition}
	+{}  {\SentenceSpace \PrintReviews} {review}
}

\makeatletter
\renewcommand{\PrintNames@a}[4]{%
	\PrintSeries{\name}
	{#1}
	{}{ and \set@othername}
	{,}{ \set@othername}
	{}{ and \set@othername}
	{#2}{#4}{#3}%
}
\makeatother

\makeatletter
\def\mathcolor#1#{\@mathcolor{#1}}
\def\@mathcolor#1#2#3{%
	\protect\leavevmode
	\begingroup
	\color#1{#2}#3%
	\endgroup
}
\makeatother
\definecolor{Red}{rgb}{0.618,0,0}
\definecolor{Blue}{rgb}{0,0,1}
\definecolor{Green}{rgb}{0,0.298,0}

\newcommand{\pvec}[1]{\vect{p}_{#1}}

\usepackage{sectsty}
\sectionfont{\color{CombinatoricaBrown}}
\subsectionfont{\color{CombinatoricaBrown}}
\subsubsectionfont{\color{CombinatoricaBrown}}

\usepackage{soul}
\soulregister\ref7

\usepackage[textsize=tiny,backgroundcolor=black!10]{todonotes}
\presetkeys{todonotes}{noline}{}

\usepackage{pifont}

\usepackage[a4paper]{geometry}
\title{\thetitle}
\author{Peleg Michaeli \and Asaf Nachmias \and Matan Shalev}

\makeatletter
\def\namedlabel#1#2{\begingroup
  #2%
  \def\@currentlabel{#2}%
  \phantomsection\label{#1}\endgroup
}
\makeatother

\usepackage{mleftright}
\mleftright
\newcommand{\midd}{\ \middle\vert\ }

\newcommand{\defn}[1]{{\bfseries #1}}

\newcommand{\eps}{\varepsilon}
\renewcommand{\phi}{\varphi}
\newcommand{\NN}{\mathbb{N}}
\newcommand{\RR}{\mathbb{R}}
\newcommand{\cB}{\mathcal{B}}

\newcommand{\cQ}{\mathcal{Q}}

\newcommand{\cT}{\mathcal{T}}
\newcommand{\stadist}{n^{-2	}}

\newcommand{\I}{\mathcal{I}}
\newcommand{\cW}{\mathcal{W}}
\newcommand{\cM}{\mathcal{M}}

\newcommand{\sB}{\mathsf{B}}
\newcommand{\sC}{\mathsf{C}}
\newcommand{\sF}{\mathsf{F}}
\newcommand{\sT}{\mathsf{T}}

\newcommand{\lr}{\leftrightarrow}

\newcommand{\floor}[1]{\left\lfloor{#1}\right\rfloor}
\newcommand{\ceil}[1]{\left\lceil{#1}\right\rceil}
\DeclareMathOperator{\polylog}{polylog}

\newcommand{\sm}{\setminus}
\newcommand{\es}{\varnothing}

\renewcommand{\d}{\deg}
\DeclareMathOperator{\diam}{diam}
\DeclareMathOperator{\dist}{dist}
\newcommand{\drat}{\hat{d}}

\newcommand{\vect}{\mathbf}

\newcommand{\pr}[0]{\mathbb{P}}
\newcommand{\E}[0]{\mathbb{E}}
\DeclareMathOperator{\proband}{and}

\DeclareMathOperator{\loops}{loops}
\newcommand{\pand}[0]{\ \proband \ }

\newcommand{\dtv}{d_{\mathrm{TV}}}
\newcommand{\ind}{\mathbf{1}}

\newcommand{\ceff}{\mathcal{C}_{\mathrm{eff}}}
\newcommand{\tmix}{t_{\mathrm{mix}}}

\newcommand{\target}{t_{\odot}}
\newcommand{\LE}{\mathsf{LE}}
\newcommand{\CT}{\mathsf{CT}}
\newcommand{\CP}{\mathsf{CP}}

\newcommand{\AB}{\mathsf{AB}}
\newcommand{\CRT}{\mathsf{CRT}}
\newcommand{\UST}{\mathsf{UST}}
\newcommand{\LERW}{\mathsf{LERW}}

\newcommand{\sun}[1]{#1^*}
\newcommand{\bub}{\cB}

\newcommand{\fut}{\mathcal{F}}
\newcommand{\past}{\mathcal{P}}
\newcommand{\height}{\mathfrak{h}}
\DeclareMathOperator{\Vol}{Vol}
\let\Cap\relax
\DeclareMathOperator{\Cap}{Cap}
\DeclareMathOperator{\Close}{Close}
\newcommand{\ball}{\mathfrak{B}}
\newcommand{\weight}{\mathbf{w}}
\newcommand{\Weight}{\mathbf{W}}
\newcommand{\partition}{\mathbf{Z}}
\newcommand{\lcevent}{F^{\star}}
\newcommand{\scevent}{F_{\star}}

\renewcommand{\r}{r}
\newcommand{\s}{s}
\renewcommand{\q}{q}

\AtEndDocument{\bigskip{\footnotesize%
\par

\textsc{Peleg Michaeli} \par
\textsc{School of Mathematical Sciences, Tel Aviv University, Tel Aviv 69978, Israel} \par
  \textit{Email:}
  \href{mailto:peleg.michaeli@math.tau.ac.il}{\texttt{peleg.michaeli@math.tau.ac.il}} \par
  \addvspace{\medskipamount}

\par

\textsc{Asaf Nachmias} \par
\textsc{School of Mathematical Sciences, Tel Aviv University, Tel Aviv 69978, Israel} \par
  \textit{Email:} 
  \href{mailto:asafnach@tauex.tau.ac.il}{\texttt{asafnach@tauex.tau.ac.il}} \par
  \addvspace{\medskipamount}

\par

\textsc{Matan Shalev} \par
\textsc{School of Mathematical Sciences, Tel Aviv University, Tel Aviv 69978, Israel} \par
  \textit{Email:}
  \href{mailto:matanshalev@mail.tau.ac.il}{\texttt{matanshalev@mail.tau.ac.il}} \par
  \addvspace{\medskipamount}

}}

\begin{document}
\maketitle

\begin{abstract}
We show that the diameter of a uniformly drawn spanning tree of a connected graph on $n$ vertices which satisfies certain high-dimensionality conditions typically grows like $\Theta(\sqrt{n})$.
In particular this result applies to expanders, finite tori $\mathbb{Z}_m^d$ of dimension $d \geq 5$, the hypercube $\{0,1\}^m$, and small perturbations thereof. 
\end{abstract}

\section{Introduction}\label{sec:intro:new}
A \defn{spanning tree} $T$ of a connected graph $G$ is a subset of edges which spans a connected graph, contains no cycles, and touches every vertex of $G$. 
The set of spanning trees of a finite connected graph $G$ is finite; hence we may consider the probability measure assigning each spanning tree equal mass.
This is known as the {\bf uniform spanning tree} ($\UST$) of $G$ denoted $\UST(G)$.
In this paper we study the distribution of the {\bf diameter} of the $\UST$, that is, the largest graph distance between two vertices.
We show that on graphs that are high-dimensional (in some precise sense given below) the diameter typically grows like the square root of the number of vertices.

When the base graph is the complete graph on $n$ vertices the limiting distribution of the diameter of the $\UST$ is typically of order $\sqrt{n}$. In fact, its limiting distribution is known~\cite{Sze83}. Moreover, the graph distance metric induced by the (properly scaled) $\UST$ converges in the Gromov--Hausdorff sense to Aldous' \emph{continuum random tree} ($\CRT$)~\cites{ACRT1,ACRT2,ACRT3}. Aldous~\cite{Ald90} was also the first to study the typical diameter of $\UST$s on general graphs and showed that if $G$ is a regular graph with a \emph{spectral gap}\footnote{The spectral gap of a graph is the difference between $1$ and the second largest eigenvalue of the transition matrix of the simple random walk on the graph.} uniformly bounded away from $0$, then the typical order of the diameter is between $\sqrt{n}/\log{n}$ and $\sqrt{n}\log{n}$ (the lower bound was later improved to $\sqrt{n}$ in~\cite{CHL12}).

Mathematical physics folklore accurately predicts that many models exhibit an {\em upper critical dimension $d_c$} above which the far away pieces of the model no longer interact.
The effect is that the geometry ``trivializes'' and to most questions on the model the answer coincides with what it would be on an infinite regular tree or the complete graph.
For the $\UST$ it is known that $d_c=4$~\cites{Pem91,BLPS01} and thus one expects that the diameter of the $\UST$ is $\sqrt{n}$ in various high-dimensional settings, such as expander graphs, finite tori of dimension $d \geq 5$, the hypercube and many more.

Indeed in these examples it can be deduced from past results (such as~\cites{Ald90, CHL12,PR04+}) that the order of the diameter is typically between $\sqrt{n}$ and $\sqrt{n}\cdot\polylog{n}$. The goal of this paper is to close the polylogarithmic gap between the upper and lower bounds, that is, to show that the diameter of the $\UST$, scaled by $n^{-1/2}$, is tight. This tightness is the first step in our future program to prove the convergence of the $\UST$ to the $\CRT$ above the critical dimension.

There is an inherent difficulty in analysing the diameter of the $\UST$.
Pemantle~\cite{Pem91} showed that the distribution of the path in the $\UST$ between two given vertices is Lawler's~\cite{Law80} {\em loop-erased random walk} ($\LERW$) between them, namely,
the simple path obtained by taking a simple random walk started at the first vertex and stopped at the second vertex, and erasing the loops of the random walk as they are created. Thus, analysing the distance between two given vertices amounts to studying the length of the $\LERW$ between them. In~\cite{PR04+}, Peres and Revelle performed a very delicate analysis of the $\LERW$ and found the precise limiting distribution of the distance (scaled by $n^{-1/2}$) between two randomly chosen vertices and showed that it matches the one given by the $\CRT$. In fact, they calculated the limiting joint distribution of the ${k\choose 2}$ distances between $k$ random vertices for any fixed $k$. However, trying to estimate the diameter from above using this result and the union bound incurs a polylogarithmic error since the corresponding events are far from disjoint. One may also try to use Wilson's algorithm~\cite{Wil96} to sample the entire $\UST$ (see \cref{sec:prelim} for the description of this algorithm). However it is unclear how to do so since the paths leading to a long diameter may be discovered at very late stages of the algorithm where it is much harder to analyse. In this paper we use a completely different technique, inspired by Peres--Revelle~\cite{PR04+} and Hutchcroft~\cite{Hut18, Hut18+}.  In particular, we use a variant of Hutchcroft's beautiful new method of sampling the $\UST$ using Sznitman's~\cite{Szn10} random interlacements process, see the proof outline in \cref{sec:proofoutline}.

\subsection{Statement of the result}
For a connected graph $G=(V,E)$ let $\drat(G)$ denote the ratio of its maximum to minimum degree.
Let $\tmix(G)$ denote the uniform mixing time of the lazy random walk on $G$ (see precise definitions in \cref{sec:prelim}).
Let $\vect{p}^t(u,v)$ denote the probability that the lazy random walk on $G$, starting from $u$, visits $v$ at time $t$, and define the \defn{bubble sum} of $G$ as follows:
\begin{equation*}
  \bub(G)
      = \sum_{t=0}^{\tmix(G)} (t+1)\sup_{v\in V} \vect{p}^t(v,v).
\end{equation*}

We will make three assumptions on $G$ with three fixed positive parameters $D$, $\alpha$ and $\theta$.
\begin{description}[style=standard,labelindent=0em,labelwidth=3em]
  \item[\namedlabel{A:bal}{(bal)}] $G$ is \defn{balanced}, namely,
    $\drat(G)\le D$.
  \item[\namedlabel{A:mix}{(mix)}] $G$ is \defn{mixing}, namely,
    $\tmix(G)  \le n^{1/2-\alpha}$.
  \item[\namedlabel{A:esc}{(esc)}] $G$ is \defn{escaping}, namely, 
    $\bub(G)\le\theta$.
\end{description}

\begin{theorem}\label{thm:main}
  For every $D,\alpha,\theta,\eps>0$ there exists $C=C(D,\alpha,\theta,\eps)$   such that if $G$ is a connected graph on $n$ vertices which satisfies \ref{A:bal}, \ref{A:mix} and \ref{A:esc} with parameters $D,\alpha,\theta$ respectively,
  then
  \begin{equation*}
    \pr\left(C^{-1}\sqrt{n}\le\diam(\UST(G))\le C\sqrt{n}\right) \ge 1-\eps \, ,
  \end{equation*}
 where $\diam(\UST(G))$ is the diameter of the $\UST$ of $G$.
\end{theorem}

Aldous conjectured~\cite[Page 460]{Ald90} that for a larger family of graphs the $\UST$ diameter is typically $\Theta(\sqrt{n})$.
It turns out that his family of graphs is slightly too wide, as we remark in \cref{sec:sharp}, where we discuss the sharpness of the three assumptions of \cref{thm:main}.
Our assumptions \ref{A:bal}, \ref{A:mix} and \ref{A:esc}, which are very similar to those of~\cite{PR04+}, try to capture the most general notion of ``dimension greater than $4$''.

\subsection{Preliminaries}\label{sec:prelim}
For a graph $G=(V,E)$ on $n$ vertices and a vertex $v\in V$ we denote by $\d_G(v)$ its degree (or simply $\d(v)$ if $G$ is clear from the context).
Denote by $\delta(G)$, $\Delta(G)$ the graph's minimum and maximum degree, respectively, and set $\drat(G)=\Delta(G)/\delta(G)$ which is only defined when $\delta(G) > 0$.
Let $X=\langle X_t\rangle$ be the \defn{lazy random walk} on $G$, that is, at each step it stays put with probability $1/2$ and otherwise moves along a uniformly chosen edge touching it. 
Since we often consider the same random walk but with different starting distributions we use the notation $\pr_\mu$ and $\E_\mu$ for probabilities and expectations conditioned on $X_0\sim\mu$ for a given probability measure $\mu$ on the vertex set. If $\mu$ is supported on a single vertex $u$, then we write $\pr_u$ or $\E_u$ instead.
For a nonnegative integer $t$ denote $\vect{p}^t(u,v)=\pr_u(X_t=v)$.

Assume that $G$ is connected and denote by $\pi$ the \defn{stationary distribution} of the lazy random walk, namely, $\pi(v) = \d(v)/2|E|$. The \defn{uniform mixing time} of the lazy random walk on $G$ is defined by 
\begin{equation*}
  \tmix(G)  = \min\left\{
    t\ge 0: \max_{u,v\in V}
      \left|\frac{\vect{p}^t(u,v)}{\pi(v)} - 1\right|\le \frac{1}{2}
  \right\}.
\end{equation*}
It is easy to see that for any $t \geq \tmix(G)$ we still have 
$$\max_{u,v\in V}
      \left|\frac{\vect{p}^t(u,v)}{\pi(v)} - 1\right|\le \frac{1}{2} \, ,$$
indeed, this follows since $\vect{p}^t(u,v)=\sum_w \vect{p}^{t-\tmix}(u,w)\vect{p}^{\tmix}(w,v)$. Hence, if $G$ is connected and satisfies \ref{A:bal} with parameter $D$, then for any $t \geq \tmix(G)$ we have 
\begin{equation}\label{fact:mix}
     \frac{1}{2Dn} \le \pr_u(X_t=v) \le \frac{2D}{n} \, .
\end{equation}
Furthermore, if $G$ satisfies in addition \ref{A:mix} and \ref{A:esc} with parameters $\alpha,\theta$ respectively, then 
\begin{equation}\label{fact:esc}
    \sum_{t=0}^{\sqrt{n}}(t+1)\sup_{v\in V}\vect{p}^t(v,v) \le \theta+2D.
\end{equation}
We will frequently appeal to \cref{fact:mix,fact:esc}.

Let us remark that due to technical convenience our definition of $\tmix$ uses the uniform distance between $\vect{p}^t(u,\cdot)$ and $\pi(\cdot)$ rather than the total variation distance that is more commonly used in the literature. However, occasionally it will be useful to use this distance. Recall that the \defn{total variation distance} between two probability measures $\mu$ and $\nu$ on $V$ is defined by
\begin{equation*}
   \dtv(\mu,\nu)
     = \sup_{A\subseteq V}|\mu(A)-\nu(A)|
     = \frac{1}{2}\sum_{v\in V}|\mu(v)-\nu(v)|.
\end{equation*}
A useful fact is that after $k\tmix(G)$ steps the total variation distance between $\vect{p}^t(u,\cdot)$ and $\pi(\cdot)$ is at most $2^{-k}$ (see, e.g.,~\cite{LPW2}*{Section 4.5}).
In particular, $t\gg\log{n}\cdot\tmix(G)$ implies
\begin{equation}\label{eq:tmix}
  \dtv(X_t,\pi) \ll n^{-2}.
\end{equation}

We will also liberally use the following asymptotic notations for two nonnegative functions $f$ and $g$ of the number of vertices $n$.
We write $f\preceq g$ if there exists a constant $C>0$ which may depend only on $D,\alpha,\theta$ such that $f(n) \le Cg(n)$ for all $n\geq 1$.  We write $f \succeq g$ if $g \preceq f$, and $f \asymp g$ if $f\preceq g$ and $f \succeq g$. We say that $f \ll g$ if $f(n)/g(n) \to 0$ as $n\to \infty$ and $f \gg g$ if $g \ll f$, and that $f(n)\sim g(n)$ if $f(n)/g(n)\to 1$ as $n\to \infty$.
For $x,y\in\RR$ we write $x\wedge y$ for the minimum between $x$ and $y$.
For the sake of clarity of presentation, we often omit floor and ceiling signs whenever they are not crucial.
We make no attempt to optimize constants that appear in our statements and arguments.

\subsubsection*{Random walks on networks and the weighted UST}
A \defn{network} $(G,\weight)$ is a graph $G=(V,E)$ with an associated \defn{weight} function $\weight:E\to\RR^+$.
The \defn{lazy random walk} on a connected network is the random walk which, given that it is at $u$, stays put with probability $1/2$, and otherwise moves to a neighbour $v$ of $u$ with probability proportional to $\weight(\{u,v\})$.

The \defn{weighted uniform spanning tree} probability measure assigns to each spanning tree $T$ a measure proportional to its weight $\Weight(T):=\prod_{e\in E(T)}\weight(e)$.
The following classical observation about the \defn{spatial Markov property} of uniform spanning trees will be useful.

\begin{claim}\label{cl:spatial}\cite{BLPS01}*{Proposition 4.2}
	Let $(G,\weight)$ be a network and $e$ be an edge of $(G,\weight)$.
	The weighted $\UST$ of $(G,\weight)$ conditioned to have $e$ as an edge is distributed as the union of $e$ with the weighted $\UST$ of the network obtained from $(G,\weight)$ by contracting the edge $e$.
\end{claim}

Given two probability measures $\mu_1$ and $\mu_2$ on $2^E$ (that is, on subsets of the edge set $E$), we say that $\mu_1$ is \defn{stochastically dominated} by $\mu_2$ if there exists a probability measure $\mu$ on $2^E\times 2^E$ which is supported on \begin{equation*}
\left\{(T_1, T_2) \in 2^E \times 2^E \mid T_1 \subseteq T_2 \right\}
\end{equation*}
and whose marginals are $\mu_1$ and $\mu_2$. We will use the following classical result proved by Feder and Mihail~\cite{FM92} (see also~\cite{LP}*{Theorem 4.6}).

\begin{lemma}\label{lem:neg:ass}
  Let $(G,\weight)$ be a network and let $A\subseteq B$ be two sets of vertices of $G$.
  Then $\UST(G/B)$ is stochastically dominated by $\UST(G/A)$.
\end{lemma}

\subsubsection*{Loop-erased random walk and Wilson's algorithm}
Consider a network $(G,\weight)$.
A \defn{walk} of length $L$ on $G$ is a sequence of vertices $(X_0,\ldots,X_L)$ such that for $0\le i<L$, the pair $\left\{X_i,X_{i+1}\right\}$ is an edge of $G$.
We say that the walk is \defn{closed} if $X_L=X_0$.
A \defn{path} (or a \defn{simple path}) is a walk which does not repeat a vertex.
Let $X=(X_0,\ldots,X_L)$ be a walk on $G$.
For a set $I\subseteq\NN$, denote $X[I]=\langle X_i\rangle_{i\in I\cap[0,L]}$, where by $X[a,b]$ we mean $X[I]$ for $I$ which is the integer interval $[\ceil{a},\floor{b}]$ (or $X[a,b)$ if we wish to exclude $b$).
The \defn{loop-erasure} of $X$, first introduced by Lawler~\cite{Law80}, denoted by $\LE(X)$, is defined as follows.
Set $\LE(X)_0 = X_0$.
For every $j\ge 1$, given $\LE(X)[0,j-1]$ we let $i = \max\{t \mid X_t = \LE\left(X\right)_{j-1}\}.$ If $i<L$, we set $\LE(X)_j = X_{i+1}$. Otherwise, $\LE(X) = \LE(X)[0,j-1]$.
In other words, we walk along the trace of $X$ and each time we close a loop we erase all edges of this loop.
If $\langle X_t \rangle $ is the lazy random walk on $(G,\weight)$ starting at $u$ and terminating when first hitting $v$, then $\LE(X)$ is called the \defn{loop-erased random walk} ($\LERW$) from $u$ to $v$.

The loop-erased random walk is fundamental for the study of the $\UST$. In 1991, Pemantle~\cite{Pem91} showed that the unique $\UST$ simple path between two vertices $u,v \in V$ is distributed as the loop-erased random walk 
from $u$ to $v$.
Wilson~\cite{Wil96} showed that the $\UST$ can be sampled by the following random procedure known today as \defn{Wilson's algorithm}. We order the vertices of $G$ in some arbitrary way $v_1, v_2,\ldots,v_n$. Let $\sT_1$ be the graph containing only $v_1$ and no edges. We now define $\sT_i$ recursively.
Given any $i\geq 2$ and $\sT_{i-1}$ we run a loop-erased random walk from $v_i$ to $\sT_{i-1}$ and obtain $\sT_i$ by appending that loop-erasure to $\sT_{i-1}$.
The \defn{times contributing to the loop-erasure} of the walk $X$ are denoted by $\langle \lambda_k(X)\rangle_{k=0}^{|\LE(X)| - 1}$ and are defined recursively by $\lambda_0(X) = 0$ and $\lambda_{k+1}(X) = 1 + \max\{t : X_t = X_{\lambda_k}\}$. 

\subsubsection*{Aldous--Broder Algorithm}\label{sec:ab}
Aldous and Broder~\cites{Ald90,Bro89}, independently, gave an algorithm to sample the $\UST$ of a connected network using random walks. Choose an arbitrary vertex $v\in V$, and start a (lazy) simple random walk from this vertex. For any vertex $u\in V \sm \{v\}$ denote by $e_v(u)$ the first edge touching $u$ in this random walk. Then the random edge set
\begin{equation*}
\{e_v(u) \mid v\in V\sm\{v\}\}
\end{equation*}
is distributed as $\UST(G)$, see~\cites{Ald90,Bro89}.

\subsubsection*{Hitting times and capacity}
Consider a network $(G,\weight)$ and the lazy random walk $X=\langle X_t\rangle$ on it.
For a nonempty vertex set $U\subseteq V$ we call the random variable
\begin{equation*}
  \tau_U = \min\left\{t\ge 0:X_t\in U\right\}
\end{equation*}
the \defn{hitting time} of $U$.  When $U=\{v\}$ we simply write $\tau_v$. For an integer $r\ge 0$, the \defn{$r$-capacity} of $U$ is the probability that a random walk starting from a stationary vertex will hit $U$ in less than $r$ steps, namely
\begin{equation}\label{eq:cap}
  \Cap_r(U) = \pr_\pi(\tau_U < r).
\end{equation}
Note that by the union bound,
\begin{equation}\label{eq:cap:ub}
  \Cap_r(U) \le r\pi(U).
\end{equation}
Also note that for any $U$ we have that $\Cap_r(U)$ is nondecreasing and subadditive in $r$.
The following fact will be useful.
\begin{claim}\label{cl:cap:mix}
  Let $u\in V$, let $U\subseteq V$ be nonempty,
  let $r\gg\log(n)\cdot\tmix(G)$ and let $t\sim r$.
  Then, for large enough $n$,
  \begin{equation*}
    \pr_u(\tau_U<t) \ge \frac{1}{3}\Cap_r(U).
  \end{equation*}
\end{claim}

\begin{proof}
	First we show that if $t'\sim r$ then 
	\begin{equation}\label{eq:half:cap}
	\Cap_{t'}(U) = \pr_{\pi}(\tau_U \leq t') \geq \frac{1}{2}\Cap_r(U).
	\end{equation}
	If $t' \geq r$ this follows from the monotonicity of the capacity.
  Assume that $t'<r$ and write $t' + m = r$. 
	By the subadditivity of the capacity we have
  \begin{equation*}
    \Cap_r(U) \leq \Cap_{t'}(U) + \Cap_m(U).
  \end{equation*}
	We may assume that $m\leq t'$, hence by monotonicity $\Cap_m(U) \leq \Cap_{t'}(U)$ and \eqref{eq:half:cap} follows.
  Let $s$ be an integer satisfying $\log(n)\cdot\tmix(G) \ll s \ll r$.
  Since $t-s\sim r$ we have that $\dtv(X_s,\pi)\ll n^{-2}$ so that we can couple (see, e.g.,~\cite{LPW2}*{Proposition 4.7}) $X_s$ with a vertex drawn according to the stationary distribution so that the coupling fails with probability at most $n^{-2}$. Thus, using \eqref{eq:half:cap} and the fact that $\Cap_r(U) \geq \pi(U) \gg n^{-2}$ we obtain that
	\begin{align*}
    \pr_u(\tau_U < t) 
    &\ge \pr_u(X[s,t) \cap U \neq \es)
    \ge \Cap_{t-s}(U) - \dtv (X_s, \pi)
     \ge \frac{1}{3}\Cap_r(U).\qedhere
	\end{align*}
\end{proof}

\subsection{High dimensional graphs}
We quickly review some notable examples of high dimensional graphs which satisfy \ref{A:bal}, \ref{A:mix} and \ref{A:esc} with universal parameters.

\paragraph{Balanced expanders}
An expander is a graph with bounded spectral gap.
The uniform mixing time of expanders is therefore logarithmic (see, e.g.,~\cite{LPW2}*{Theorem 12.4}).
In balanced expanders we have, by~\cite{LPW2}*{Eq.~(12.13)}, that $\vect{p}^t(u,v)$, for any two vertices $u,v$ and $t\ge 0$, is bounded from above by $c_1 n^{-1} + c_2 e^{-c_3t}$, for some constants $c_1,c_2,c_3>0$, from which \ref{A:esc} is immediately deduced.

\paragraph{Tori of dimension $5$ and above}
The mixing time of a $d$-dimensional tori on $n$ vertices is at most of order $n^{2/d}$ (see, e.g.,~\cite{LPW2}*{Theorem 5.6}), so for $d\ge 5$ we have \ref{A:mix}.
In addition, for every vertex $v$ and $t\leq \tmix$, we have $\vect{p}^t(v,v)\leq Ct^{-d/2}$ for some constant $C>0$ (see, e.g.,~\cite[Theorem 2.3]{DSC94}), hence \ref{A:esc} follows.

\paragraph{The hypercube}
The mixing time of the hypercube on $n$ vertices is $\asymp\log{n}\cdot\log{\log{n}}$ (see, e.g.,~\cite{LPW2}*{Section 5.3.1}).
In addition, in the nonlazy random walk $\vect{p}^6(u,u)\leq C\log^{-3}{n}$ for some constant $C>0$ and $\vect{p}^{2t+1}(v,v) = 0$ for every integer $t$.
Since in the nonlazy random walk $\vect{p}^{2t}(v,v)$ is nonincreasing in $t$ and the probability that the lazy random walk of length $t$ makes less than $6$ nonlazy steps decays exponentially in~$t$, we obtain \ref{A:esc}.

\subsection{Sharpness of \cref{thm:main}}\label{sec:sharp}
It is easy to construct a graph which does not satisfy the 
assumptions of \cref{thm:main} but whose $\UST$ is of diameter of order $\sqrt{n}$, for example, the graph formed by connecting two disjoint cliques on $n/2$ vertices by a single edge.
Nevertheless,
the following examples demonstrate that relaxing any of the assumptions of \cref{thm:main} may result in a $\UST$ with diameter which is not of order $\sqrt{n}$.

Let $H$ be a $3$-regular expander on $n/\log{n}$ vertices and consider the graph $G$ obtained by attaching a simple path of length $\log{n}$ to each of the vertices of $H$.
This graph satisfies \ref{A:bal} and \ref{A:mix} but not \ref{A:esc}.
The diameter of a spanning tree of this graph is $2\log{n}$ plus the diameter of the tree restricted to $H$.
Since the order of the diameter of $\UST(H)$ is typically $\sqrt{n/\log(n)}$ this will be the order of the diameter of $\UST(G)$.
The necessity of \ref{A:bal} is justified by looking at the star graph which has diameter $2$ with probability $1$ and satisfies both \ref{A:mix} and \ref{A:esc}.
To see that \ref{A:mix} is necessary, just consider the path on $n$ vertices (In this example \ref{A:esc} does not hold as well. We do not know of an example in which \ref{A:bal} and \ref{A:esc} both hold, but not \ref{A:mix}). \\

A slightly different attempt to capture the high-dimensionality assumptions which guarantee mean field behaviour is due to Aldous~\cite[Page 460]{Ald90}. Recall that $\tau_v$ is the hitting time of a vertex $v$. 
For a vertex $u\in V$, let
\begin{equation*}
\target^u = \sum_{v\in V}\pi(v)\E_u(\tau_v)
\end{equation*}
be the \defn{target time} of $u$.  It can be shown (see, e.g., 
\cite{LPW2}*{Lemma 10.1}) that it does not depend on the choice of $u$. 
The \defn{relaxation time} of $G$ is the inverse of its spectral gap 
and is related to the mixing time and to the Cheeger constant of the graph (see, e.g.,~\cite{LPW2}*{Theorem 12.4}).

Aldous conjectured that for every regular graph $G$ on $n$ vertices with linear target time and polylogarithmic relaxation time
the expected diameter of $\UST(G)$ is of order $\sqrt{n}$.
However, the class of graphs satisfying these conditions turns out to be slightly too large, as we show in the following counterexample.

Let $\Gamma_n$ be the graph obtained from $\mathbb{Z}_{\log{n}}^4$ by taking two distinguished nonadjacent vertices which will be chosen later,
removing $6$ edges emanating from each of them and adding a perfect matching between the $6$ other endpoints of the removed edges. Let $G$ be obtained from a $4$-regular $n$-vertex transitive expander $H$ by replacing each edge of $H$ with a copy of $\Gamma_n$ and mapping the endpoints of the edge to the two distinguished vertices of $\Gamma_n$.
Note that $G$ is $8$-regular.
We first show that this graph satisfies the assumption of Aldous and then estimate the diameter of $\UST(G)$ using a result of Schweinsberg~\cite{Sch09} and some basic properties of expanders. This graph has polylogarithmic expansion hence by Cheeger's inequality (see~\cite{LPW2}*{Theorem 13.10}) we obtain that the relaxation time is also polylogarithmic. To bound the target time, we recall that both in a regular expander and in the $4$-dimensional torus~\cite{LPW2}*{Proposition 10.21}, the effective resistance between any two vertices is bounded above by a constant.
A moment's reflection shows that the effective resistance in $G$ between any two vertices is also bounded by a constant. Therefore, by the commute time identity~\cite{LPW2}*{Proposition 10.7}, for every $u,v\in V(G)$ we have that $\E_u[\tau_v]$ is at most linear in the number of vertices of $G$ and therefore the target time is at most linear as well.

We will now estimate from below the diameter of $\UST(G)$. 
There exists a coupling of $\UST(H)$ and $\UST(G)$ such that if we condition on $\{e_1,\ldots,e_k\} \subseteq \UST(H)$, then the subtrees obtained by restricting $\UST(G)$ to each of the copies of $\Gamma_n$ associated with the edges $\{e_i\}_{i=1}^k$ are i.i.d.\ with distribution $\UST(\Gamma_n)$. We condition on the edges of $\UST(H)$. By~\cite{PR04+} and using this coupling, the path between two typical vertices in $\UST(G)$ passes through $\asymp \sqrt{n}$ copies of $\Gamma_n$ with high probability. Using~\cite{Sch09}, there exist two vertices in $\Gamma_n$ such that the typical distance between them in $\UST(\Gamma_n)$ is of order $\log^2{n}(\log{\log{n}})^{1/6}$ with probability bounded away from $0$. These are the two points we choose as the distinguished points of $\Gamma_n$ in the construction described in the previous paragraph. By standard large deviations for i.i.d.~variables (e.g., Chernoff bound) we have that with high probability in a constant fraction of the copies of $\Gamma_n$ along the path between two typical vertices of $\UST(G)$ the distance between the two distinguished vertices of $\Gamma_n$ is of order $\log^2{n}(\log{\log{n}})^{1/6}$. This yields that the typical distance between two vertices in $G$ is with high probability of order at least $\sqrt{n\log^4{n}}(\log{\log{n}})^{1/6}$, rather than the conjectured order of $\sqrt{n\log^4{n}}$.

\subsection{Proof outline}\label{sec:proofoutline}
The first step in our proof, performed in \cref{sec:lb}, is to sample the $\UST$ path $\phi$ between two random vertices $u$ and $v$ and prove that it is typically of length of order $\sqrt{n}$ and that it is not too ``tightly packed''.
This is the statement of \cref{thm:lb}.
The lower bound on the length of $\phi$ immediately gives the required lower
bound on the diameter of \cref{thm:main}. A difficulty that arises is that
the typical length of the random walk's path between the endpoints of $\phi$
is of order $n$, while the length of its loop-erasure is typically much
shorter, namely, of order $\sqrt{n}$.
This was dealt with in~\cite{PR04+} and we borrow their idea as follows. We extend our graph by adding an extra vertex which we call the ``sun'' and connect it to every vertex of the graph with a weighted edge. It is useful since the forest obtained from the $\UST$ of the ``sunny'' graph by removing any edge touching the sun is stochastically dominated by the $\UST$ of the original graph. The time it takes the random walk to visit the sun is geometric with mean of order $\sqrt{n}$. We run a random walk from $u$ until it hits the sun and take its loop-erasure; then we take another random walk from $v$ until it hits the first loop-erasure and argue that these two paths are unlikely to meet at the sun. Thus we are able to sample the path $\phi$ by running random walks of length at most $\sqrt{n}$ and use the domination property to bound the length of $\phi$ from above and below. 

We also need to show that with high probability $\phi$ is sufficiently ``spread out''. This will be useful for the upper bound on the diameter, see below.
We will prove that the $\r$-capacity of $\phi$, when $\tmix(G)\ll r \ll \sqrt{n}$, is of order $rn^{-1/2}$ with high probability, i.e., it has a lower bound which matches the upper bound given in \eqref{eq:cap:ub} up to multiplicative constants. This roughly means that $\phi$ looks like a random set of vertices. 
The path $\phi$ is of length $\sqrt{n}$ which is much longer than the mixing time (by \ref{A:mix}), hence different chunks of $\phi$ are almost independent of each other and occur roughly at uniform and independent locations in $G$. We carefully study the structure of the cutpoints of $\phi$ --- these typically take a positive fraction of $\phi$ (as we expect in dimensions above $4$) --- so they alone suffice to constitute the lower bound on the capacity.

For the upper bound we condition on $\phi$ and consider the order $\sqrt{n}$ trees hanging on it in the $\UST$. We prove that the probability that each tree reaches height $\ell$ is of order at most $\ell^{-1}$. Taking $\ell = A \sqrt{n}$ for some large constant $A$ together with the union bound immediately gives the upper bound on the diameter. To show the $\ell^{-1}$ estimate above we use a variant of Hutchcroft's~\cite{Hut18} construction of the wired uniform spanning forest where the set $\phi$ plays the role of ``$\infty$''. Recall that the Aldous--Broder algorithm (see \cref{sec:ab}) ran on the trace of an infinite random walk starting from $\phi$ (which we contracted to a single vertex) outputs a $\UST$ of the graph. We may simulate this random walk by considering a Poisson process on the space $\cW\times\RR$ where $\cW$ is the space of random walk excursions from $\phi$ to itself. The Aldous--Broder algorithm performed on the concatenation of the excursions which occur after a fixed time $t\in \RR$ has the distribution of the $\UST$. Thus we have constructed a stationary process distributed as the $\UST$ at each time. This construction allows us to produce a recursive inequality similar to~\cite{Hut18+} and in the spirit of~\cite{KN09} which yields the required $\ell^{-1}$ bound, see the details in \cref{sec:recursiveinequality}. The fact that $\phi$ is sufficiently spread-out, is crucially used throughout the proof. For example, it implies that the length of an excursion from $\phi$ to itself is comparable to the length of its loop-erasure.

\section{Lower bound}\label{sec:lb}
The goal of this section is to prove the lower bound in \cref{thm:main}.
To do so we will show that the unique simple path in the $\UST$ between two independent stationary points is of order at least $\sqrt{n}$.
In fact, for proving the upper bound (in \cref{sec:ub}) it will be of importance to show that this path also has large capacity (meaning, informally, that it is not too ``ball-like''), but that it is not of order larger than $\sqrt{n}$.
This path is distributed as the $\LERW$ between these vertices.
A difficulty in analysing this $\LERW$ is that typically most of the original walk, which is of linear length, does not survive the loop-erasure. We therefore begin by analysing loop-erasures of much shorter random walks in which a positive fraction of the walk typically survives the loop erasure. For that, a couple of definitions and claims about properties of short $\LERW$s are needed.
 
Throughout this section, $G=(V,E)$ is a connected graph on $n$ vertices satisfying \ref{A:bal}, \ref{A:mix} and \ref{A:esc} with parameters $D,\alpha,\theta$, respectively. We assume that the number of vertices $n$ is large enough.
Define the \defn{buffer time} $\s$ and the \defn{run time} $\r$ to be
\begin{equation}\label{eq:rs}
  \s:=n^{1/2-2\alpha/3}, \qquad \r:= n^{1/2-\alpha/3},
\end{equation}
which we will use in the rest of this paper.
Note that $\s\ll\r\ll\sqrt{n}$ and $\s\gg\log(n)\cdot\tmix(G)$.
Let us now state the main theorem of this section.
\begin{theorem}\label{thm:lb}
  For every $\eps>0$ there exist $\chi,A>0$ depending only on $D,\alpha,\theta,\eps$ such that the following holds.
  Let $u,v$ be two independent stationary points in $G$, and let $\phi$ be the unique simple path between them in $\UST(G)$.
  Then,
  \begin{equation*}
    \pr\left(|\phi|\le A\sqrt{n}\pand \Cap_\r(\phi) \ge \chi\r n^{-1/2}\right) \ge 1-\eps.
  \end{equation*}
\end{theorem}

We remark that the lower bound on the capacity of $\phi$ in \cref{thm:lb} gives, in view of~\eqref{eq:cap:ub}, the needed lower bound on its length.
The next two claims show that two independent short random walks starting from the same stationary vertex have positive probability to never intersect.
\begin{claim}\label{cl:nonatt}
  Let $u$ be a stationary vertex and let $X,Y$ be two independent lazy random walks starting at $u$.
  Then, for every vertex~$v$,
  \begin{equation*}
    \E\left[ \left| \left\{
      (i,j): 0\le i,j\le \r\pand X_i=Y_j=v
    \right\} \right| \right]
    \preceq n^{-1}.
  \end{equation*}
\end{claim}

\begin{proof}
  For any fixed $i,j\ge 0$ and $v\in V$, by reversibility and by \ref{A:bal},
  \begin{align*}
    \sum_{u\in V}\pi(u)\cdot \pr(X_i=Y_j=v \mid X_0 = Y_0 = u)
    &= \sum_{u\in V}\pi(u)\cdot \pr_u(X_i=v)\cdot\pr_u(Y_j=v)\\
    &\le \frac{D}{n}\sum_{u\in V}\pr_v(X_i=u)\cdot\pr_u(Y_j=v)
     = \frac{D}{n}\vect{p}^{i+j}(v,v).
  \end{align*}
  thus by reversibility and \cref{fact:esc},
  \begin{align*}
    \E\left[ \left| \left\{
      (i,j): 0\le i,j\le \r\pand X_i=Y_j=v
    \right\} \right| \right]
    &\le \sum_{t=0}^{2 \r} \sum_{i=0}^t \sum_{u\in V} \pi(u)
      \cdot \pr_u(X_i=Y_{t-i}=v)\\
    &\le \frac{D}{n} \sum_{t=0}^{2 \r} \sum_{i=0}^t \vect{p}^t(v,v)
     \le \frac{D(\theta+2D)}{n}.\qedhere
  \end{align*}
\end{proof}

\begin{claim}\label{cl:escape}
  Let $u$ be a stationary vertex and let $X,Y$ be two independent lazy random walks starting at $u$.
  Then,
  \begin{equation*}
    \pr(X[0,\r ]\cap Y[1,\r ] = \es) \succeq 1.
  \end{equation*}
\end{claim}

\begin{proof}
  First note that by \ref{A:bal}, \ref{A:mix}, \cref{fact:mix} and the union bound,
  \begin{equation*}
    \pr(\exists i,j:\ 0< i,j\le 2\r,\ \max\{i,j\}>\r,\ X_i=Y_j)
    \le 6 \r ^2 D/n \ll 1.
  \end{equation*}
  Call a pair $(i,j)$, for $0\le i,j\le \r $, a \defn{last intersection}, if $X_i=Y_j$ and for every $i\le t_1\le 2\r $ and $j< t_2\le 2\r $ it holds that $X_{t_1}\ne Y_{t_2}$.
  Thus, with probability greater than $1/2$, there exists a last intersection, thus the expected number of last intersections is greater than $1/2$.
  On the other hand,
  \begin{align*}
    \pr((i,j)\text{ is a last intersection})
    &\le \sum_{v\in V} \pr(X[0,\r ]\cap Y[1,\r ]=\es \mid X_0=Y_0=v)\cdot \pr(X_i=Y_j=v),
  \end{align*}
  thus by \cref{cl:nonatt} and by \ref{A:bal},
  \begin{align*}
    \frac{1}{2} &\le \E[\left|\left\{
      (i,j):(i,j)\text{ is a last intersection}
    \right\}\right|]\\
    &\le \sum_{v\in V} \pr(X[0,\r ]\cap Y[1,\r ]=\es \mid X_0=Y_0=v)
         \cdot \sum_{i,j=0}^{\r } \pr(X_i=Y_j=v)\\
    &\preceq n\cdot\sum_{v\in V} \pi(v) \cdot \pr(X[0,\r ]\cap Y[1,\r ]=\es\mid X_0=Y_0=v)
         \cdot n^{-1}\\
    &= \pr(X[0,\r ]\cap Y[1,\r ]=\es),
  \end{align*}
  and the claim follows.
\end{proof}

\subsection{Cut points and capacity}
Let $X = (X_0,\ldots,X_L)$ be a fixed walk of length $L$. We say that an integer $0\le t< L$ is a \defn{cut time} of $X$ if $X[0,t] \cap X[t+1,L] = \es$.
We observe that if $t$ is a cut time of $X$, then $X_t \in \LE(X[0,L])$.
Denote the set of cut times of $X$ by $\CT(X)$ and let $\CP(X)=X[\CT(X)]$ be the set of \defn{cut points} of $X$.
For every $1\le i\le\floor{L/\r}$ define
\begin{equation}\label{eq:Ai:Bi}
  B_i(X) = X[(i-1)\r,ir - \s), \quad A_i(X) = X[(i-1)\r+\s, ir - 2\s).
\end{equation}

The next claim essentially states that the capacity of the loop-erasure of a random walk whose length is a small constant times $\sqrt{n}$ is large enough.
\begin{claim}\label{cl:cap:sqrt}
  There exist a nonnegative function $f(n,\beta)$ satisfying $\lim_{\beta\to 0}\lim_{n\to\infty}f(n,\beta)=0$ and a constant $c>0$, both depending only on $D$, $\alpha$ and $\theta$, such that the following holds.
  Let $X$ be a lazy random walk of length $\beta\sqrt{n}$ starting from a stationary vertex and set $N=\floor{\beta\sqrt{n}/\r}$.
  Then with probability at least $1-f(n,\beta)$ we have
  \begin{equation*}
    \Cap_\r\left(
        \bigcup_{i=1}^N \left(\CP(B_i(X))\cap A_i(X)\right)
      \right) \ge \frac{c\beta \r}{\sqrt{n}}
    \quad\pand\quad
      \bigcup_{i=1}^N \left(\CP(B_i(X))\cap A_i(X)\right) \subseteq \LE(X),
  \end{equation*}
	so that
  \begin{equation*}
    \pr(\Cap_\r(\LE(X)) \geq c\beta \r/\sqrt{n}) \geq 1 - f(n,\beta).
  \end{equation*}
\end{claim}

The idea of the proof is as follows.
We first show that with high probability there are no loops of size larger than $\s$ in $X$.
Then, we divide the walk into small runs of size $\r-\s$, separated by buffers of size $\s$ (these are the sets $B_i(X)$ defined in \eqref{eq:Ai:Bi}).
We look at the cut times of any of these runs.
If indeed there are no loops of size larger than $\s$, then every cut time which is not in the $\s$-long prefix or suffix of the run survives the loop-erasure of $X$.
Then, in order to estimate the capacity of this loop-erasure,
we show that the sum of the capacities of the sets of cut points in these runs is typically of the right order.
Since with high probability these runs are ``far'' from one another, this sum estimates well the capacity of the loop-erasure of $X$.
This is the content of the next two claims.

To ease notation, set
\begin{equation*}
  \q = \r/\sqrt{n} = n^{-\alpha/3}.
\end{equation*}

\begin{claim}\label{cl:cap:r}
  Let $X$ be a lazy random walk of length $r-1$ starting from a stationary vertex.
  Then \[\E[\Cap_\r(\CP(X))] \succeq q^2.\]
\end{claim}

\begin{proof}
	Let $Y$ be a lazy random walk starting from a stationary vertex, independent of $X$.
	We then have
	\begin{equation*}
	\E[\Cap_\r(\CP(X))] = \pr(\CP(X) \cap Y \neq \es).
	\end{equation*}
	For $0\le i,j < \r$ let $I_{ij}$ be the indicator of the event $\{X_i=Y_j\}$ and let $J_{ij}$ be the indicator of the event $\{I_{ij}\pand X_i\in\CP(X)\}$.  Let $I=\sum_{i,j}I_{ij}$ and $J=\sum_{i,j}J_{ij}$. 
  By \cref{cl:escape}, 
  for every $0\le i<\r$ we have $\pr(i\in \CT(X)) \succeq 1$ (where the implicit constant does not depend on $i$). Hence, using \ref{A:bal},
  \begin{align*}
  	\E[J_{ij}]
    &\ge \sum_{v\in V} \pr(X_i = Y_j = v \pand i\in \CT(X))\\
    &=   \sum_{v\in V} \pr(X_i = Y_j = v) \pr(i\in \CT(X) \mid X_i = Y_j = v)\\
  	&\ge \frac{1}{Dn} \sum_{v\in V} \pr(X_i = v) \pr(i\in \CT(X) \mid X_i = v)
     \succeq n^{-1}, 
  \end{align*}
  and it follows that $\E{J} \succeq \q^2$. 
  We turn to bound the second moment of $J$.  Observing that $0\le J\le I$, it suffices to bound the second moment of $I$ instead.
  Suppose that $i',j'$ are such that $|i'-i|=k_i$ and $|j'-j|=k_j$.  Then 
  \begin{align*}
  	\E[I_{ij}I_{i'j'}]
    &= \sum_{u,v\in V} \pr(X_i=Y_j=u \pand X_{i'} = Y_{j'} = v) \\
    &\le \sum_{u\in V} \pr(X_i=Y_j=u)
         \sum_{v\in V} \pr(X_{i'} = Y_{j'} = v \mid X_i = Y_j = u) \\
    &\le \sum_{u\in V} \pr(X_i = Y_j = u)
         \sum_{v\in V} D\vect{p}^{k_i}(u,v)\vect{p}^{k_j}(v,u)
     \leq D\sum_{u\in V} \pi^2(u) \vect{p}^{k_i+k_j}(u,u). 
  \end{align*}
  We sum over all $i,j,i',j'$ to obtain
  \begin{align*}
  	\E[J^2] \le \E[I^2] \le 4 D \r^2 \sum_{u\in V} \pi^2(u) \sum_{t=0}^{2r} 
  	(t+1)\vect{p}^{t}(u,u). 
  \end{align*}
  From \cref{fact:esc} we obtain that the second sum is bounded by 
  $\theta+2D$. We then have that 
  \begin{equation*}
  	\E[J^2] \le 4Dr^2(\theta+2D) \sum_{u\in V} \pi^2(u)
    \preceq q^2.
  \end{equation*}
  Hence, by Paley--Zigmund inequality,
  \begin{equation*}
  	\pr(J>0) \ge \frac{\E^2[J]}{\E[J^2]}
    \succeq q^2.
  \end{equation*}
  The result follows since $\pr(J>0) = \E[\Cap_\r(\CP(X))]$.
\end{proof}

Define the \defn{$\r$-closeness} of two vertex sets $U_1,U_2$ by \[\Close_\r(U_1,U_2) = \pr_\pi(\tau_{U_1} < \r \pand \tau_{U_2} < \r).\]

\begin{claim}\label{cl:close}
  Let $Y$ and $Z$ be two independent lazy random walks of length $r-1$ starting from two independent stationary vertices.  Then,
  \begin{equation*}
    \E[\Close_\r(Y,Z)]\preceq \q^4.
  \end{equation*}
\end{claim}

\begin{proof}
    Let $X$ be a lazy random walk, independent of $Y$ and $Z$, starting from an independent stationary vertex. Throughout the proof, the notation $\tau$ refers to the random walk $X$.  We would like to bound $\pr(\tau_{Y} < \r \pand \tau_{Z} < \r)$.
	Indeed, by symmetry,
	\begin{equation*}
	  \pr(\tau_{Y} < \r \pand \tau_{Z} < \r)
    \le 2\pr(\exists t_1<t_2 <\r : X_{t_1} \in Y,\ X_{t_2} \in Z). 
	\end{equation*}
  Conditioning on $Z$ and partitioning according to the hitting vertex in $Y$ we have
	\begin{equation*} \label{eq:closeness:fixed:set}
	  \pr(\exists t_1<t_2 < \r : X_{t_1} \in Y,\ X_{t_2} \in Z \mid Z)
    \le \sum_{v\in V}
          \pr(\tau_v = \tau_{Y} < \r) \pr_v(\tau_{Z} < \r \mid Z).
	\end{equation*}
	By the union bound, $\pr(\tau_v = \tau_{Y} < \r) \le \r^2\pi^2(v)$,
	hence, using \ref{A:bal},
	\begin{align*}
    \sum_{v\in V}\pr(\tau_v = \tau_{Y} < \r)\pr_v(\tau_{Z} < \r \mid Z)
    &\le \sum_{v\in V} \r^2\pi^2(v)\pr_v(\tau_{Z}<\r \mid Z)\\
    &\le \r^2\cdot\frac{D}{n}\sum_{v\in V} \pi(v)\pr_v(\tau_{Z} < \r \mid Z)
     \le Dq^2\cdot \Cap_\r(Z),
	\end{align*}
  and the result follows since by~\eqref{eq:cap:ub} we have $\Cap_\r(Z)\le r\pi(Z)\preceq q^2$.
\end{proof}

To prove \cref{cl:cap:sqrt}, we will also need the following classical lemma by Hoeffding.
\begin{lemma}[Hoeffding,~\cite{Hoe63}*{Theorem 1}]\label{lem:hoeff}
 	Let $M>0$, and suppose $\{J_i\}_{i=1}^k$ is a family of independent random variables with $0 \leq J_i \le M$. Then, for every $0<x<M-\E[\sum_{i=1}^k J_i]/k$,
  \begin{equation*}
  	\pr\left(\left|\sum_{i=1}^k  \left(J_i  - \E[J_i]\right) \right| > kx\right)
      \le 2\exp\left(-2k\left(\frac{x}{M}\right)^2\right).
  \end{equation*}	
\end{lemma}

\begin{proof}[Proof of \cref{cl:cap:sqrt}]
  We begin by bounding the probability that there is a loop of size larger than $\s$ in $X$.
  This follows from \cref{fact:mix} and the union bound:
  \begin{equation}\label{eq:large:loop:estimate}
    \pr\left(\exists i,j: |i-j|>\s,\ X_i = X_j\right)
      \le (\beta\sqrt{n})^2\cdot 2D/n \preceq \beta^2.
  \end{equation}
  Therefore, with probability tending to $1$ as $\beta$ tends to $0$ we obtain
  \begin{equation*}
    \bigcup_{i=1}^N \left(\CP(B_i(X)) \cap A_i(X)\right) \subseteq \LE(X).
  \end{equation*}
  The sequence of runs $\langle B_i(X)\rangle$ are nearly independent random walks with initial distribution $\pi$. Indeed, because they are separated by buffers of length $\s$, if we condition on $B_i(X)$ the distribution of $X_{(i+1)\r}$, the starting position of $B_{i+1}(X)$, is $\stadist$-far from $\pi$ in total variation distance.
  Define the process $Y$ as follows.
  For $0\le i<N$ we sample $Y_{ir}\sim\pi$ independently, and let $\langle Y_{ir+j}\rangle_{j=0}^{r-1}$ be a lazy random walk on $G$.
  That is, $Y$ is a concatenation of $N$ independent lazy random walk starting from the stationary distribution.
  By the discussion above we may couple $X$ and $Y$ such that
  \begin{equation}\label{eq:coupling:success}
    \pr(\forall i: B_i(Y) = B_i(X))
      \ge 1 - N\cdot n^{-2}.
  \end{equation}
  That is, if we neglect an error which, for $n$ large enough, is bounded by $N\cdot n^{-2}$, we can get our estimations on $\langle B_i(X)\rangle$ by analysing $Y$, which has the property that $\langle B_i(Y) \rangle$ are i.i.d.\ lazy random walks on $G$ starting from the stationary distribution. 
  We now show that the sum of $\Cap_\r(\CP(B_i(X))\cap A_i(X))$ over $i$ is large, using the coupling between $X$ and $Y$.
  We use Hoeffding's lemma (\cref{lem:hoeff}) with the random variables $J_i = \Cap_\r(\CP(B_i(Y))\cap A_i(Y))$.
  By \cref{cl:cap:r} we have that $\E[J_i] \succeq \q^2$.
  By~\eqref{eq:cap:ub} we have $J_i \le D\q^2$.
  We obtain
  \begin{equation*}
    \pr\left(
        \left|\sum_{i=1}^N \left(J_i - \E[J_i]\right) \right| > N\q^{9/4}
      \right)
    \le 2\exp\left(-2N \left(\frac{\q^{9/4}}{2Dr^2n^{-1}}\right)^2 \right)
    \le 2\exp\left(-c_1\beta \q^{-1/2}\right),
  \end{equation*}
  for some $c_1=c_1(D,\alpha,\theta)$.
  Therefore, with probability larger than $1- 2\exp\left(-c_1\beta \q^{-1/2}\right)$ we have
  \begin{equation}\label{eq:sum:cap:estimate}
    \sum_{i=1}^N \Cap_\r\left(\CP({B_i(Y)})\cap(A_i(Y)\right)
      \succeq \beta\q.
  \end{equation}
  We would like to show that this sum estimates well enough (from below) the capacity of $\LE(X)$. By a simple inclusion-exclusion argument, we have 
  \begin{equation*}
    \Cap_\r\left(\bigcup_{i=1}^N\CP(B_i(X))\cap A_i(X))\right)
      \ge \sum_{i=1}^N \Cap_\r(\CP(B_i(X))\cap A_i(X))
        - \sum_{i\ne j} \Close_\r(B_i(X),B_j(X)).
  \end{equation*}
  Therefore, on the event that there are no loops of size larger than $\s$,
  \begin{equation*}
    \Cap_\r(\LE(X))
      \ge \sum_{i=1}^N \Cap_\r(\CP(B_i(X))\cap A_i(X))
        - \sum_{i\ne j} \Close_\r(B_i(X),B_j(X)).
  \end{equation*}
  We use the coupling between $X$ and $Y$ once more and bound from above $\Close_\r(B_i(Y),B_j(Y))$ for $i\neq j$.
  It follows from \cref{cl:close} that
  $\E[\Close_\r(B_i(Y),B_j(Y))] \preceq \q^4$.
  Hence
  \begin{equation*}
    \E\left[\sum_{i\ne j} \Close_\r(B_i(Y),B_j(Y))\right]
      \preceq N^2\q^4 \asymp \beta^2\q^2.
  \end{equation*}
  Then using Markov's inequality,
  \begin{equation}\label{eq:close:estimate}
    \pr\left(\sum_{i\ne j} \Close_\r(B_i(Y),B_j(Y)) > \q^{3/2} \right)
      \preceq \beta^2\q^{1/2}.
  \end{equation}
  We sum the errors by restricting to walks with no loops of size larger than $\s$ using~\eqref{eq:large:loop:estimate}, in which the coupling succeeds using~\eqref{eq:coupling:success}, with large capacity as in~\eqref{eq:sum:cap:estimate} and small closeness as in~\eqref{eq:close:estimate} to obtain that with probability $1-f(n,\beta)$, for $f(n,\beta)$ which tends to $0$ as $n\to\infty$ and $\beta\to 0$,
  we have that
  \begin{equation*}
      \Cap_\r(\LE(X))
        \ge \Cap_\r\left(\bigcup_{i=1}^N \CP(B_i(X)) \cap A_i(X)\right)
        \succeq \beta\q.\qedhere
  \end{equation*}
\end{proof}

\subsection{The sunny network}\label{sec:sun}
Let $\Cap_\r(\cdot)$ denote the $\r$-capacity as defined in~\eqref{eq:cap}, where $\r=\r(\alpha)$ is as defined in~\eqref{eq:rs}.
For $\beta>0$, let $\sun{G}_\beta:=(\sun{G},\weight_\beta)$ be the network obtained from the graph $G$ in the following way.
First, add a vertex $\rho$, called the \defn{sun}, to the vertex set of $G$. Then, for every $u\in V$ we add the edge $\{u,\rho\}$ to $\sun{G}$, assigning it edge weight so that the probability to move to $\rho$ from $u$ in the lazy random walk would be $\beta^2n^{-1/2}$.
This is achieved by setting
\begin{equation*}
  \weight_\beta(\{u,\rho\}) = \frac{2\beta^2 \d_G(u)}{\sqrt{n}-2\beta^2}
\end{equation*}
and $\weight_\beta(\{u,v\})=1$ for every $\{u,v\}\in E$.
We call $\sun{G}_\beta$ the \defn{sunny network} associated with $G$ (with parameter $\beta$).
Note that if $X$ is the lazy random walk on $\sun{G}_\beta$, conditioning on $\tau_\rho=m$ for some $m\ge 1$, the walk $X[0,m)$ is distributed as the lazy random walk of length $m-1$ on $G$.

A classical fact~\cite{LP}*{Lemma 10.3} says that if $G$ is a connected subnetwork of $\sun{G}_\beta$, then we can couple the weighted $\UST$ measure on $G$ and the weighted $\UST$ measure on $\sun{G}_\beta$ such that the weighted $\UST$ of $\sun{G}_\beta$, restricted to the edge set of $G$, is a subset the weighted $\UST$ of $G$.
Let $\sun\phi$ be the unique simple path between two stationary points $u,v$ in $\UST(\sun{G}_\beta)$.
We show that as $\beta$ tends to $0$ the probability that $\sun\phi$ contains the sun diminishes, hence by this coupling $\sun\phi$ is the unique simple path between $u$ and $v$ in $\UST(G)$.
The general approach is to run a random walk from $u$ until it hits the sun, analyse the length and capacity of its loop-erasure, then run a random walk from $v$ until it hits that loop-erasure, and analyse its loop-erasure as well.

\begin{claim}\label{cl:pathtosun}
  There exist a nonnegative function $f(n,\beta)$ satisfying $\lim_{\beta\to 0}\lim_{n\to\infty}f(n,\beta)=0$ and a constant $c>0$, both depending only on $D$, $\alpha$ and $\theta$, such that the following holds.
  Let $X$ be a lazy random walk on $\sun{G}_\beta$ starting from a stationary vertex of $G$ and terminated $1$ step before first hitting $\rho$.
  Then with probability at least $1-f(n,\beta)$ we have
  \begin{equation*}
    \Cap_\r(\LE(X)) \geq c\beta \r / \sqrt{n}
      \pand |\LE(X)| \le \sqrt{n}/\beta^3. 
  \end{equation*}
\end{claim}

\begin{proof}
  We first note that since $\tau_\rho$ (the first hitting time to $\rho$) is a geometric random variable with mean $\sqrt{n}/\beta^2$, we have that $\pr(\tau_\rho \leq \beta \sqrt{n}) \le \beta^3$.
  We restrict ourselves to the event $\{\tau_\rho > \beta \sqrt{n}\}$.
  Denote $m=\tau_\rho - \beta\sqrt{n} - 1$ and divide the walk to a prefix $X[0,m)$ and a $\beta\sqrt{n}$-long suffix $X[m,\tau_\rho)$.
  We distinguish between two cases according to whether or not the capacity of the loop-erasure of the prefix is larger than $2c\beta\r / \sqrt{n}$, where $c=c(D,\alpha,\theta)$ is the constant obtained from \cref{cl:cap:sqrt}.
  Write
  \begin{equation*}\label{eq:prefix:long:cap}
    \lcevent = \left\{\Cap_\r(\LE(X[0,m))) \geq 2c\beta\r / \sqrt{n}
                      \pand \tau_{\rho} \geq \beta\sqrt{n}\right).
  \end{equation*} First, we suppose that $\lcevent$ holds.
  In this case, we will find a part of the loop-erasure of the prefix that the suffix avoids.
  This part will thus survive the loop-erasure and contribute $c\beta\r /\sqrt{n}$ to its capacity.
  Let $t$ be the smallest integer such that $\LE(X[0,m))[0,t]$ has capacity larger than $c\beta\r/\sqrt{n}$.
  For brevity, denote $U = \LE(X[0,m))[0,t]$.
  By the subadditivity of the capacity we have that
  \begin{equation*}
    \Cap_\r(U)
      \le \Cap_\r(\LE(X[0,m))[0,t)) + \Cap_\r(\LE(X[0,m))_t)
      \le c\beta\r / \sqrt{n} + Dr / n \leq 2c\beta\r / \sqrt{n}.
  \end{equation*}
  Recall that $\langle \lambda(X[0,m))_i \rangle$ are the times contributing to the loop-erasure of $X[0,m)$.
  The subadditivity of the capacity also implies that the length of $\LE(X[0,m))$ has to be larger than $t-1+c\beta\sqrt{n}/D$.
  Hence on $\lcevent$, for large enough $n$,
  \begin{equation*}
    \lambda(X[0,m))_t \leq m - \frac{c\beta\sqrt{n}}{2D}.
  \end{equation*}
  We will now show that with high probability $U$ does not intersect the suffix $X[m,\tau_\rho)$ and hence survives the loop-erasure of $X[0,\tau_\rho)$. We apply \ref{A:mix} and use the fact that conditioned on $\tau_\rho$, the walk $X[0,\tau_{\rho})$ is distributed as the lazy random walk on $G$.
  We then have by \cref{fact:mix}
  \begin{align*}
    &\pr(\lcevent \pand U \cap X[m,m+s) \neq \es) \\
    &\leq \pr(X[m,m+\s) \cap X[0,m-c\beta\sqrt{n}/(2D)] \neq \es) \\
    &\leq \E[\ind(X[m,m+\s) \cap X[0,m-c\beta\sqrt{n}/(2D)] \neq \es) \mid m]
      \leq \frac{2D\s\E[\tau_\rho]}{n}
      \leq \frac{2D\s}{\beta^2 \sqrt{n}}.
  \end{align*}
  Therefore there is a low probability that $U$ intersects $X[m,m+s)$.
  We will show it has a low probability of intersecting $X[m+s,\tau_{\rho})$. 
  By \eqref{eq:tmix}, every $\r$-long part of the walk $X[m+\s,\tau_\rho)$ has a probability of at most $\Cap_\r(U) + n^{-2}$ to hit $U$.
  We divide $X[m+\s,\tau_\rho)$ into $\ceil{(\beta\sqrt{n}-s)/r}$ sub-runs of length $\r$.
  Conditioning on $\lcevent$, the expected numbers of such runs that intersect $U$ is bounded by $(\beta\sqrt{n}/\r)\cdot(3c\beta\r/\sqrt{n})=3c\beta^2$.
  Then, 
  \begin{align*}
    \pr(\lcevent \pand \Cap_\r(\LE(X)) < c\beta \r / \sqrt{n})
    &\leq \pr(\lcevent \pand U \not\subseteq \LE(X)) \\
    &\leq 2Ds/(\beta^2\sqrt{n}) + 3c\beta^2.
  \end{align*}
  
  We now analyse the second case. Denote
  \begin{equation*}
  \scevent = \left\{\Cap_\r(\LE(X[0,m))) < 2c\beta \r / \sqrt{n} \pand \tau_{\rho} \geq \beta\sqrt{n}\right).
  \end{equation*}
  On $\scevent$ a lazy random walk of length $\beta \sqrt{n}$ starting from a stationary vertex hits $\LE(X[0,m))$ with probability at most $2c\beta^2$. 
  Since the distribution of $X_{m+\s}$ conditioned on $X[0,m)$ is close to stationarity (as in \eqref{eq:tmix}) we have
  \begin{equation} \label{eq:lb:suffix:hit:small:prefix}
    \pr(\scevent \pand X[m+\s,\tau_\rho) \cap \LE(X[0,m)) \neq \es) \leq 2c\beta^2 + n^{-2}.
  \end{equation}
  Moreover, the probability that there is a loop of size at least $\s$ in the last $\beta \sqrt{n}$ steps of the walk is smaller than $D\beta^2$.
  Set $N=\floor{(\beta\sqrt{n}-\s)/\r}$.
  Suppose $\scevent$ holds but the event in~\eqref{eq:lb:suffix:hit:small:prefix} does not hold.
  Suppose further that there is no loop of size $s$ in the suffix.
  In that case, for every $1\le i\le N$ we have that
  \begin{equation*}\label{eq:lb:cp:in:lerw}
    \CP(B_i(X[m+\s,\tau_\rho))) \cap A_i(X[m+\s,\tau_{\rho})) \subseteq \LE(X).
  \end{equation*}
  By \cref{cl:cap:sqrt}, there exists a function $f_1(n,\beta)$ satisfying $\lim_{\beta\to 0}\lim_{n\to\infty}f_1(n,\beta)=0$ 
  depending only on $D$, $\alpha$ and $\theta$, such that
  \begin{equation*} \label{eq:lb:suffix:has:large:cap}
    \pr\left(
	    \Cap_\r\left(
        \bigcup_{i=1}^N \CP(B_i(X[m+\s,\tau_\rho)))\cap A_i(X[m+\s,\tau_\rho))
		  \right) \geq \frac{c\beta \r}{\sqrt{n}}
    \right)
    \geq 1 - f_1(n,\beta).
  \end{equation*}
  We sum all errors in the two cases to obtain
  \begin{equation*}
  \pr\left(\Cap_\r(\LE(X)) < c\beta \r / \sqrt{n}\right)
  \leq
    2Ds/(\beta^2\sqrt{n}) + 3c\beta^2
    + D\beta^2 + 2c\beta^2 + n^{-2} + f_1(n,\beta).
  \end{equation*}
  For the upper bound on the length of $\LE(X)$, we use Markov's inequality to obtain
  \begin{equation*}
    \pr(|\LE(X)| \geq \sqrt{n}/\beta^3)
      \le \pr(\tau_\rho \geq \sqrt{n}/\beta^3) \leq \beta.
  \end{equation*}
  Therefore with probability $1-f(n,\beta)$ for $f(n,\beta)$ which satisfies $\lim_{\beta\to 0}\lim_{n\to\infty}f(n,\beta)=0$ we have that
  $\Cap_\r(\LE(X)) \succeq \beta \r/\sqrt{n}$ and $|\LE(X)|\le\sqrt{n}/\beta^3$.
\end{proof}

We now know that the loop-erased random walk from a stationary vertex to the sun behaves nicely, and we turn to analyse the loop-erasure of an independent random walk starting from an independent stationary vertex that terminates when it hits the first loop-erasure.

\begin{claim}\label{cl:pathtopath}
  There exist a nonnegative function $f(n,\beta)$ satisfying $\lim_{\beta\to 0}\lim_{n\to\infty}f(n,\beta)=0$ and a constant $c>0$, both depending only on $D$, $\alpha$ and $\theta$, such that the following holds.
  Let $X$ be a lazy random walk on $\sun{G}_\beta$ starting from a stationary vertex of $G$ and terminated when first hitting $\rho$.
  Let $Y$ be an independent lazy random walk on $\sun{G}_\beta$ starting from an independent stationary vertex of $G$ and terminated when hitting $\LE(X)$.
  Then with probability at least $1-f(n,\beta)$ we have
  \begin{equation*}
    Y_{\tau_{\LE(X)}} \neq \rho \ \pand\ \Cap_\r(\LE(Y)) \geq c\beta^4r/\sqrt{n}\ \pand\ |\LE(Y)| \leq \sqrt{n}/\beta^3.
  \end{equation*}
\end{claim}

\begin{proof}
  We denote $W = \LE(X)\sm\{\rho\}$ for brevity.
  By \cref{cl:pathtosun} we have that $\Cap_\r(W)\ge c\beta rn^{-1/2}$ and $|W|\le\beta^{-3}n^{1/2}$ with probability $1-f_2(n,\beta)$, for a constant $c>0$ and a function $f_2$ which satisfies $\lim_{\beta\to 0}\lim_{n\to\infty} f_2(n,\beta)=0$, both depending only on $D$, $\alpha$ and $\theta$.
  We condition on this event and we let $\tau^Y_W$ be the hitting time of the random walk $Y$ in a set $W$.
  Using \cref{cl:cap:mix}, for every $j\geq 0$ we have
  \begin{equation*} \label{eq:avoid:set:capacity}
  \pr(\tau_W^Y > jr) \leq \left(1-\frac{c\beta r}{3\sqrt{n}}\right)^j.
  \end{equation*} 
  Conditioning on $\{\tau_W > j\r\}$, for every $t> jr$ we have $\pr(X_t=\rho) = \beta^2/\sqrt{n}$. 
  We then have
  \begin{align*}
    \pr\left(\tau^Y_\rho < \tau^Y_W\right)
    &\leq \sum_{j=0}^\infty \pr(jr \leq \tau_\rho^Y \leq (j+1)r \pand \tau_W^Y > jr) \\&\leq \sum_{j=0}^{\infty}
      \pr\left(jr < \tau^Y_\rho \leq (j+1)\r \midd \tau^Y_W > jr\right)
      \cdot \pr(\tau^Y_W > jr) \\
    &\leq \sum_{j=0}^{\infty}
      \frac{\beta^2r}{\sqrt{n}}
      \cdot\left(1-\frac{c\beta \r}{3\sqrt{n}}\right)^j 
     \preceq \beta.
  \end{align*}
  We will now show that $\Cap_\r(\LE(Y)) \succeq \beta^4r / \sqrt{n}$ with high probability.
  Indeed, by the union bound, $\pr(\tau^Y_{\LE(X)} \leq \beta^4\sqrt{n}) \le \beta^4\sqrt{n}\cdot D|W|/n + \beta^6 \le D\beta$.
  By Markov's inequality 
  \begin{equation*}
  \pr(\tau^Y_{\LE(X)} > \sqrt{n}/\beta^3) \leq \pr(\tau^Y_\rho > \sqrt{n}/\beta^3) \leq \beta.
  \end{equation*}  
   We thus restrict ourselves to the event 
  \begin{equation}\label{eq:length:bound:second:walk}
    \{\beta^4\sqrt{n} <\tau^Y_{\LE(X)} \le \beta^{-3}\sqrt{n}\}.
  \end{equation}
  We analyse the walk $Y[0,\beta^4 \sqrt{n}]$. By \cref{cl:cap:sqrt}, we have that $\Cap_\r(\LE(Y[0,\beta^4\sqrt{n}])) \succeq \beta^4r/\sqrt{n}$ with probability larger than $1-f_1(n,\beta)$, for $f_1$ which satisfies $\lim_{\beta\to 0}\lim_{n\to\infty} f_1(n,\beta)=0$.
  We would like to show that most of $\LE(Y[0,\beta^4\sqrt{n}])$ survives the loop-erasure of $Y[0,\tau_{\LE(X)}]$.
  On the event in~\eqref{eq:length:bound:second:walk}, the probability that $Y[0,\beta^4\sqrt{n} - \s]$ intersects $Y[\beta^4\sqrt{n},\tau^Y_{\LE(X)}]$ is bounded by 
  \begin{equation*}
    \E\left[
      \#\left\{(i,j) \midd Y_i = Y_j,
        \ 0\leq i \leq \beta^4\sqrt{n} -\s,
        \ \beta^4\sqrt{n}\leq j \leq \tau^Y_{\LE(X)}
      \right\}\right]
    \leq D\beta.
  \end{equation*}
  Finally, the expected number of loops of size larger than $\s$ in the first $\beta^4 \sqrt{n}$ is smaller than $D\beta^8$.
  When all these events occur, we have that $\LE(Y[0,\beta^4\sqrt{n}-\s])\cap Y[0,\beta^4\sqrt{n}] \subseteq \LE(Y)$.
  We therefore have that with probability larger than $1-f(n,\beta)$ that $\Cap(\LE(Y)) \geq c\beta^4r/\sqrt{n}$, that  $Y_{\tau_{\LE(X)}} \neq \rho$ and $|\LE(Y)| \leq \sqrt{n}/\beta^3$, for $f$ which satisfies $\lim_{\beta\to 0}\lim_{n\to\infty} f(n,\beta)=0$.
\end{proof}

We are now ready to prove the main theorem of this section.

\begin{proof}[Proof of \cref{thm:lb}]
  Let $\eps>0$.
  Let $\sun{G}_\beta$ be the sunny network associated with $G$ (with parameter $\beta$, to be chosen later). We couple $\UST(G),\UST(\sun{G}_\beta)$ as described in the beginning of the section.
  Sample two independent stationary points $u$ and $v$ from $V$, and let $\phi$ and $\sun\phi$ be the simple paths between them in $\UST(G),\UST(\sun{G}_\beta)$, respectively.
  It follows that $\phi=\sun\phi$ if and only if $\rho\notin\sun\phi$.
  By \cref{cl:pathtosun,cl:pathtopath} there exist $c=c(D,\alpha,\theta)>0$ and $\beta=\beta(\eps)>0$ such that for large enough $n$,
  \begin{equation*}
    \pr\left(\phi=\sun\phi
             \pand |\sun\phi|\le 2\sqrt{n}\beta^{-3}
             \pand \Cap_\r(\sun\phi)\ge c\beta^4rn^{-1/2}\right)
    \ge 1-\eps.
  \end{equation*}
  The result follows by taking $A=2\beta^{-3}$ and $\chi=c\beta^4$.
\end{proof}

\begin{proof}[Proof of the lower bound in \cref{thm:main}]
  Let $\eps>0$, let $u,v$ be two independent stationary points and let $\phi$ be the unique simple path between them in $\UST(G)$.
  By~\eqref{eq:cap:ub} and \cref{thm:lb} we have that with probability at least $1-\eps$,
  \begin{equation*}
    \diam(\UST(G)) \ge |\phi|
    \ge \frac{\Cap_\r(\phi) \cdot n}{\r D}
    \ge \frac{\chi\sqrt{n}}{D}.\qedhere
  \end{equation*}
\end{proof}

\section{Upper bound}\label{sec:ub}
Let $G=(V,E)$ be a connected graph on $n$ vertices, satisfying \ref{A:bal}, \ref{A:mix} and \ref{A:esc} with parameters $D,\alpha,\theta$ respectively.
The goal of this section is to prove the upper bound of \cref{thm:main}.
To do so, we will show that under certain size and capacity assumptions on a given set of vertices $W$, the diameter of the $\UST$ of $G/W$ --- the graph obtained from $G$ by contracting $W$ into a single vertex --- is of order at most $\sqrt{n}$.
This will be helpful to bound the diameter of $\UST(G)$ by setting $W$ to be the set of vertices in the unique simple path in $\UST(G)$ between two independent stationary vertices which satisfies, according to \cref{sec:lb}, the necessary assumptions.

Let $\Cap_\r(\cdot)$ denote the $\r$-capacity as defined in~\eqref{eq:cap}, where $\r=\r(\alpha)$ is as defined in~\eqref{eq:rs}.

\begin{theorem}\label{thm:ub}
  For every $\chi,A,\eps>0$ there exists $C=C(D,\alpha,\theta,\chi,A,\eps)$ such that if $W$ is a vertex set satisfying $|W|\le A\sqrt{n}$ and $\Cap_r(W)\ge \chi r/\sqrt{n}$ then
  \begin{equation*}
    \pr(\diam(\UST(G/W)) \ge C\sqrt{n}) \le \eps.
  \end{equation*}
\end{theorem}

For a nonempty vertex set $W$, denote
\begin{equation*}
  \pvec{W}^t(u,v) = \pr(X_t = v \pand X[0,t] \cap W = \es \mid X_0=u),
\end{equation*}
and define the $W$-\defn{bubble sum} by
\begin{equation*}
  \bub_W(G) := \sum_{t=0}^{\infty}(t+1)\sup_{v\in V}\pvec{W}^t(v,v).
\end{equation*}
As the random walk on $G$ is an irreducible Markov chain on a finite state space, we have that $\pr(X[0,t]\cap W = \es)$ decays exponentially in $t$. Hence, $\bub_W(G)$ is finite.

\begin{claim}\label{cl:bubw:walk:bd}
  Let $W$ be a nonempty vertex set.
  Let $\ell\ge 1$ be an integer and let $\langle \gamma_t \rangle_{t=0}^\ell$ be a simple path of length $\ell$, such that $v=\gamma_0,\ldots,\gamma_{\ell-1}\notin W$ and $\gamma_\ell=u\in W$.
  Then, for a lazy random walk $X$ from $v$ to $W$,
  \begin{equation*}
    \E[\tau_u \mid \LE(X[0,\tau_W]) = \gamma,\ \tau_u = \tau_W] \leq \ell\cdot\bub_W(G).
  \end{equation*}
\end{claim}

\begin{proof}
  We set $\loops_k^\gamma(W)$ to be the set of loops (closed walks) starting and ending at $\gamma_k$ that do not hit $\gamma[0,k-1] \cup W$. 
  Let $\langle Y_t\rangle$ be the lazy random walk on $G$ and consider the following weight $\weight$ assigned to every finite loop $\phi \in \loops_k^\gamma(W)$,
  \begin{equation*}
    \weight(\phi)
      = \pr_{\gamma_k}(Y[0,|\phi|]=\phi)
      = \prod_{i=0}^{|\phi|-1}\vect{p}^1(\phi_i,\phi_{i+1}).
  \end{equation*}
  Denote
  \begin{equation*}
    \partition_k := \sum_{\phi \in \loops_k^\gamma(W)}\weight(\phi).
  \end{equation*}
  The quantity $\partition_k$ equals to the expected number of visits to $\gamma_k$ before hitting $\gamma[0,k-1]\cup W$. As the random walk on $G$ is an irreducible Markov chain on a finite state space, $\partition_k$ is finite.
  We observe that $\partition_k \geq 1$ since the loop of length $0$ contributes $1$ to $\partition_k$.
  Recall that $\langle \lambda_i(X[0,\tau_W])\rangle_i$ are the times contributing to the loop-erasure of $X[0,\tau_W]$.
  Conditioning on $X[0,\lambda_k(X[0,\tau_W])]$, on $\LE(X[0,\lambda_k(X[0,\tau_W])])=\gamma[0,k]$ and on $\tau_u = \tau_W$, the conditional distribution of $X[\lambda_k(X[0,\tau_W]),\tau_u]$ is equal to the distribution of a random walk started at $X_{\lambda_k(X[0,\tau_W])}$ conditioned to hit $u$ before it hits any other vertex of $W$ or any vertex of $\gamma[0,k-1]$.
  Therefore,
  \begin{align*}
    &\pr_v(X[\lambda_k,\lambda_{k+1}-1] = \phi
             \mid \LE(X[0,\tau_W]) = \gamma,\ \tau_u = \tau_W) \\
    &= \frac{\weight(\phi)}{\partition_k}\cdot \ind_{\{\phi \in \loops_k^\gamma(W)\}}
      \leq \weight(\phi)\cdot \ind_{\{\phi \in \loops_k^\gamma(W)\}}.  
  \end{align*}
  Summing over all $\phi$ of length $m$ we have that
  \begin{equation*}
    \pr_v\left(\lambda_{k+1}(X[0,\tau_u]) - \lambda_k(X[0,\tau_u]) - 1
    = m \midd \LE(X[0,\tau_W]) = \gamma,\ \tau_u = \tau_W\right)
    \leq \sup_{v\in V}\pvec{W}^m(v,v).
  \end{equation*}
  Hence,
  \begin{align*}
    &\E_v\left[\lambda_{k+1}(X[0,\tau_u]) - \lambda_k(X[0,\tau_u])
                 \mid \LE(X[0,\tau_W]) = \gamma,\ \tau_u = \tau_W\right]\\
    &\leq\sum_{m=0}^{\infty}(m+1)\sup_{v\in V}\pvec{W}^m(v,v)=\bub_W(G).
  \end{align*}
  We conclude that
  \begin{align*}
    &\E_v[\tau_u \mid \LE(X[0,\tau_W]) = \gamma,\ \tau_u =  \tau_W] \\
    &= \sum_{k=0}^{\ell-1}
         \E_v\left[\lambda_{k+1}(X[0,\tau_u]) - \lambda_k(X[0,\tau_u])
                     \mid \LE(X[0,\tau_W]) = \gamma,\ \tau_u = \tau_W\right]
    \leq \ell\cdot\bub_W(G).\qedhere
  \end{align*}
\end{proof}

Recall that $G/W$ is the graph obtained from $G$ by contracting the vertex set $W$ into a single vertex.  We think of the edge set of $G/W$ as the same edge set as that of $G$, by remembering the original endpoints of each edge (self loops may be created).
We think of $\UST(G/W)$ as an \emph{oriented} random tree, in which the edges are oriented towards $W$. 
The graph $\cT_W=(V,E(\UST(G/W)))$ which consists of the edges of $\UST(G/W)$ on the vertices of $G$ is therefore an oriented rooted forest (a disjoint union of oriented rooted trees).
If $W=\{w_1,\ldots,w_k\}$ we may write $\cT_{w_1,\ldots,w_k}$ instead. 

For any oriented rooted forest $T$ and a vertex $v$ denote by $\past^T(v)$ the \defn{past} of $v$ in $T$, namely, the tree spanned by the set of vertices that have a directed path \emph{to} $v$ in $T$. Similarly, denote by $\fut^T(v)$ the \defn{future} of $v$ in $T$, namely, the path spanned by the set of vertices that have a directed path \emph{from} $v$ in $T$.

For any graph $G$, a vertex $v$ in $G$ and an integer $\ell\ge 0$, denote the graph distance ball of radius $\ell$ around $v$ in $G$ by $\ball_G(v,\ell)$.

\begin{claim}\label{cl:hang:ball}
  Let $W$ be a nonempty vertex set and $u\in W$.  Then for every integer $\ell\ge 1$,
  \begin{equation*}
    \E\left[\left|\ball_{\cT_W}(u,\ell)\right|\right] \leq 8D\ell\cdot\bub_W(G).
  \end{equation*}
\end{claim}

\begin{proof}
	Let $v\in V\setminus W$.
  Let $X$ be a random walk on $G$ starting from $v$ and let 
  $J^\ell(u)$ be the event that $\tau_u = \tau_{W}$ and that $|\LE(X[0,\tau_{W}])| \leq \ell$.
  By Wilson's algorithm, $\pr_v(J^\ell(u))$ is the probability that $v$ is in $\ball_{\cT_W}(u,\ell)$.
  For every path $\gamma$ such that ${}\pr_v(\LE(X[0,\tau_W]) = \gamma \pand \tau_u = \tau_W) > 0$ we have, by \cref{cl:bubw:walk:bd} and Markov's inequality, that
  \begin{equation*}
    \pr_v\left(
      \tau_u > 2|\gamma|\cdot\bub_W(G)
        \midd \LE(X[0,\tau_W]) = \gamma,\ \tau_u=\tau_{W}
    \right)
    \leq \frac{1}{2}.
  \end{equation*}
  Averaging over all $\gamma$ with $|\gamma|\leq\ell$ we get that
  \begin{equation*}
    \pr_v\left(\tau_u > 2\ell\cdot\bub_W(G)
               \midd J^\ell(u) \right) \leq \frac{1}{2}.
  \end{equation*}
  Thus, 
  \begin{align*}
    \pr_v(J^\ell(u))
    &\leq 2 \pr_v(\tau_u=\tau_W \pand \tau_u \leq 2\ell\cdot\bub_W(G)) \\
    &\leq 2\sum_{t=0}^{2\ell\cdot\bub_W(G)}\pr_v(\tau_u=\tau_W=t)\\
    &\leq 2D
      \sum_{t=0}^{2\ell\cdot\bub_W(G)}\pr_u(\tau_v = t < \tau_W^+),
  \end{align*}
  where the last inequality follows by reversing time.
  Summing over all $v$ we obtain
  \begin{align*}
    \E\left[\left|\ball_{\cT_W}(u,\ell)\right|\right]
    &= 1 + \sum_{v\in V\sm W}\pr_{v} (J^\ell(u))\\
    &\le 1 + 2D\sum_{t=0}^{2\ell\cdot \bub_W(G)}\sum_{v\in V\sm W}\pr_u(\tau_v=t<\tau_W^+)
     \le 8D\ell\cdot\bub_W(G).\qedhere
  \end{align*}
\end{proof}

\subsection{Effective conductance}
For every two disjoint sets $W$ and $S$, we define the \defn{effective conductance} between them to be
\begin{align*}
\ceff(W \lr S) = \Vol(W)\cdot \pr_W(\tau_S<\tau_W^+)
= 2|E|\cdot\sum_{u\in W} \pi_u\cdot\pr_u(\tau_S<\tau_W^+),
\end{align*}
where $\Vol(W)=\sum_{v\in W}\d(v)$ is the \defn{volume} of $W$,
and $\pr_W(\cdot)$ denotes the probability measure of the lazy random walk on $G/W$ conditioned on $X_0=W$.
Note that reversibility ensures that $\ceff(W\lr S)=\ceff(S\lr W)$.
Define further
\begin{equation*}
\cM_W(S) = \sum_{u,v\in S}\d(u)\mathcal{G}_W(u,v),
\end{equation*}
where $\mathcal{G}_W(u,v)$ is the \defn{Green function} of the random walk killed when hitting $W$. That is, 
\begin{equation*}
\mathcal{G}_W(u,v)
= \E_u\left[\sum_{t=0}^{\infty}\ind_{\{X_t = v \pand t < \tau_W\}}\right].
\end{equation*}
Note that $\cM_W(S)$ is monotone increasing in $S$.
We now use the following lemma from~\cite{Hut18+}*{Lemma 5.5}.
\begin{lemma}\label{lem:ceff:bd}
	Let $W,S \subseteq V$ be two disjoint vertex sets.  Then
	\begin{equation*}
	\ceff(W \lr S) \geq \frac{\left(\sum_{v\in S}\d(v)\right)^2}{\cM_W(S)}.
	\end{equation*} 	
\end{lemma}

For a random walk $X$ and $t\ge 0$, let
\[X_W[0,t] = \langle X_i \rangle_{i=\tau_{W^c}}^{t \wedge (\tau_{W^+} - 1)},\]
with the convention that if $\tau_{W^c} > \tau_{W^+} - 1$ then $X_W[0,t]$ is empty.
Note that $X_W[0,t] \cap W = \es$. 
The next lemma is analogous to~\cite{Hut18+}*{Lemma 5.6} and is proven using the same techniques. 
\begin{lemma}\label{lem:m:exp}
	Let $W$ be a nonempty vertex set.
	Let $v\in V$ and let $X$ be a lazy random walk. Then,
	\begin{equation*}
	\E_v[\cM_W(X_W[0,t])] \leq 5\Delta(G)\bub_W(G)t.
	\end{equation*}
\end{lemma}
\begin{proof}
	Let $v\in V\setminus W$ and
	denote $\Delta=\Delta(G)$.
	Let $Y^i$ be an independent lazy random walk starting from $X_i$.
	We then have
	\begin{equation*}
	\E_v[\cM_W(X_W[0,t])]
	= \sum_{i=0}^{t}\sum_{j=0}^{t}\sum_{k=0}^{\infty}
	\E_v[\d(X_i)\ind(Y_k^i = X_j\pand (Y^i[0,k]\cup X[0,j])\cap W = \es)].
	\end{equation*}
	As in~\cite{Hut18+}, we divide this sum into two parts, according to whether or not $i\leq j$.  When $i\leq j$ we upper bound by 
	\begin{align*}
	&\sum_{i=0}^{t}\sum_{j=i}^{t}\sum_{k=0}^\infty\sum_{u,w\in V}
	\vect{p}_W^i(v,u)\d(u)\vect{p}_W^{j-i}(u,w)\vect{p}_W^k(u,w)\\
	&\leq \Delta
	\sum_{i=0}^{t}\sum_{j=i}^{t}\sum_{k=0}^\infty\sum_{u,w\in V}
	\vect{p}_W^{i}(v,u)\vect{p}_W^{j-i}(u,w)\vect{p}_W^k(w,u)\\
	&\leq \Delta
	\sum_{i=0}^{t}\sum_{j=i}^{t}\sum_{k=0}^\infty\sum_{u\in V}
	\vect{p}_W^{i}(v,u)\vect{p}_W^{k+j-i}(u,u)\\
	&\leq \Delta
	\sum_{i=0}^{t}\sum_{j=0}^{i}\sum_{k=0}^{\infty}\sup_u
	\vect{p}_W^{k+j-i}(u,u).
	\end{align*}
	Similarly, for the second sum, we upper bound by
	\begin{equation*}
	\Delta\sum_{i=0}^{t}\sum_{j=0}^{i-1}\sum_{k=0}^{\infty}\sup_u
	\vect{p}_W^{i-j+k}(u,u).
	\end{equation*}
	Therefore,
	\begin{equation*}
	\E_v[\cM_W(X_W[0,t])]
	\leq \Delta       
	\sum_{i=0}^{t}\sum_{j=0}^{t}\sum_{k=0}^{\infty}\sup_u 
	\vect{p}_W^{|j-i|+k}(u,u).
	\end{equation*}
	For every pair $b,m$ with $b\le m$ we have at most $2t+1$ suitable values of $i,j,k$ satisfying $|j-i|=b$ and $b+k=m$.  Hence,
	\begin{align*}
	\E_v[\cM_W(X_W[0,t])]
	&\leq \Delta(2t+1)
	\sum_{m=0}^{\infty}\sum_{b=0}^{m\wedge t}\sup_u
	\vect{p}_W^m(u,u) \\
	& \leq \Delta(2t+1)\left(\sum_{m=0}^t (m+1)\sup_u
	\vect{p}_W^{m}(u,u)
	+ \sum_{m=t+1}^{\infty}(t+1)\sup_u
	\vect{p}_W^{m}(u,u)\right) \\
	& \leq \Delta(2t+1)(2\bub_W(G))
	\leq 5\Delta\bub_W(G)t.
	\end{align*} 	
	This proves the lemma for every $v\in V\sm W$. For $v \in W$, we have 
	\begin{align*}
	\E_v[\cM_W(X_W[0,t])]
	&\le \sum_{u\in V\sm W} \vect{p}(v,u) \E_u[\cM_W(X_W[0,t])]\\
	&\le \sum_{u\in V\sm W} \vect{p}(v,u) 5\Delta\bub_W(G)t
	\le 5\Delta\bub_W(G)t. \qedhere
	\end{align*}
\end{proof}

Recall that $\fut^W(v)$ is the future of $v$ in $\cT_W$.
Denote by $\hat\fut^W(v)$ the future of $v$ in $\cT_W$ excluding the unique vertex of $W$ in this future, namely, $\hat\fut^W(v)=\fut^W(v)\sm W$.
The following lemma is an analogue of~\cite{Hut18+}*{Lemma 7.7}.

\begin{lemma}\label{lem:bad:bd}
	Let $W$ be a nonempty vertex set.
	Let $\xi>0$ and let $\ell>0$ be an integer. 
	Then,
	\begin{equation*}\label{eq:hut:low:cap:ball:bound}
	\sup_{u\in W}\E\left[\left|\left\{
	v\in V \midd \dist_{\cT_W}(u,v) \in 
	\left[\frac{\ell}{3},\frac{2\ell}{3}\right]
	\pand 
	\ceff\left(W \lr \hat\fut^W(v)\right) \leq \xi\ell\Delta(G)
	\right\}\right|\right] \leq C_1\xi\ell,
	\end{equation*}	
	for $C_1 = 360D^3\bub_W(G)^3$.
\end{lemma}

\begin{proof}
	Write $\delta=\delta(G)$ and $\Delta=\Delta(G)$.
	Fix $v\in V\sm W$ and $u\in W$.  For a random walk $X$, let $Z=X[0,\tau_u-1]$.
	Let $F_1$ be the event $\{|\LE(Z)|\in[\ell/3,2\ell/3]\}$.
	Let $F_2$ for the event $\{\tau_u=\tau_W\}$ and $F_2'$ be the event $\{\tau_u=\tau_W\le 2\bub_W(G)\ell\}$.
	Finally, let $F_3$ be the event $\{\ceff(W\lr\LE(Z)) \le \xi\ell\Delta\}$.
	Then
	\begin{equation*}
	\pr\left(\dist_{\cT_W}(u,v) \in [\ell/3,2\ell/3]
	\pand \ceff(W \lr \hat\fut^W(v)) \leq \xi\ell\Delta\right)
	= \pr(F_1\cap F_2\cap F_3).
	\end{equation*}
	By \cref{cl:bubw:walk:bd} and the same reasoning as in \cref{cl:hang:ball} we have that this probability is bounded by
	$2\pr(F_1\cap F_2'\cap F_3)$.
	Since on the event $F_1\cap F_2'\cap F_3$ we have $|\LE(Z)| \geq \ell/3$, it follows from \cref{lem:ceff:bd} that
	\begin{equation*}
	\frac{\ell^2\delta^2}{9\cM_W(Z)}
	\leq \frac{\ell^2\delta^2}{9\cM_W(\LE(Z))}
	\leq \ceff(W \lr \LE(Z)) \leq \xi \ell\Delta. 
	\end{equation*}
	Thus, on this event we have $\cM_W(Z) \geq \frac{\Delta\ell}{9D^2\xi}$.
	Note also that on this event $Z=X_W[0,\tau_u]$.
	Hence, by reversing paths,
	\begin{align*}
	\pr_v(F_1\cap F_2\cap F_3)
	&\le 2\pr_v(F_1\cap F_2'\cap F_3)\\
	&\leq 2\pr_v\left(F_2' \pand
	\cM_W(X_W[0,\tau_u]) \geq \frac{\Delta\ell}{9D^2\xi}\right)\\
	&\le 2\sum_{t=0}^{2\bub_W(G)\ell}
	\pr_v\left(t=\tau_W = \tau_u
	\pand \cM_W(X_W[0,t]) \geq \frac{\Delta\ell}{9D^2\xi}\right)\\
	& \leq 2D\sum_{t=0}^{2\bub_W(G)\ell}
	\pr_u\left(X_t = v,\  \tau_W^+ > t
	\pand \cM_W(X_W[0,t]) \geq  \frac{\Delta\ell}{9D^2\xi}\right).
	\end{align*}
  To conclude, we sum over all $v$ in $V\sm W$ and apply \cref{lem:m:exp} and Markov's inequality to
  obtain that this quantity is bounded by
	\begin{equation*}
	2D\sum_{t=0}^{2\bub_W(G)\ell}
	\pr_u\left(\cM_W(X_W[0,t]) \geq \frac{\Delta\ell}{9D^2\xi}\right)
	\leq 90D^3\xi\cdot \bub_W(G)
	\sum_{t=0}^{2\bub_W(G)\ell} \frac{t}{\ell}
	\leq 360D^3\bub_W(G)^3 \xi\ell.\qedhere
	\end{equation*}
\end{proof}

\subsection{Aldous--Broder and random interlacements}\label{sec:AB:RI}
The random interlacement process was first introduced by Sznitman~\cite{Szn10}. Furthermore, many of the ideas in this section are based on~\cite{Hut18}.
Given a finite connected graph $G$ and a nonempty subset $W\subseteq V$ we define the \defn{$W$-wired interlacement process} $\I_W$ as follows.
Let $\cW_W$ be the set of walks $(u_i)_{i=0}^\ell$ satisfying
\begin{itemize}
	\item $\ell\ge 1;$
	\item $u_0,u_\ell\in W;$
	\item $u_i\notin W$ for $0<i<\ell$.
\end{itemize}
We call these walks \defn{$W$-trajectories} and think of them as directed walks.
We equip $\cW_W$ with the following probability measure:
\begin{equation*}
  \mu_W\left((u_i)_{i=0}^\ell\right)
  = \frac{1}{\Vol(W)} \cdot \prod_{i=1}^{\ell-1}\frac{1}{\d(u_i)},
\end{equation*}
where we recall that $\Vol(W)=\sum_{v\in W}\d(v)$.
The $W$-wired interlacement process $\I_W$ is the Poisson point process on $\cW_W\times\RR$ with
intensity measure $\mu_W\otimes\mathcal{L}$, where $\mathcal{L}$ is the Lebesgue 
measure on $\RR$.  We think of $\RR$ as the timeline.  We write
$\I_W[a,b]$ for the intersection of $\I_W$ with $\cW_W\times[a,b]$ and $\I_W(t)$ for the intersection of $\I_W$ with $\cW_W\times[t,\infty)$.

For a vertex $v\in V$ let $\sigma_t(v;W)$ be the first time greater or 
equal to $t$ at which there is a directed edge emanating from $v$ in a $W$-trajectory of $\I_W(t)$. We observe that for $v\in V\sm W$ this is exactly the first time greater or equal to $t$ at which $v$ is contained in a $W$-trajectory of $\I_W(t)$ and for a vertex $v\in W$ this is the first time greater or equal to $t$ at which there is a $W$-trajectory starting from $v$. 
Since such a first trajectory is almost surely unique, we may define $e_t(v;W)$, for every $v\in V\sm W$, to be the first directed edge of that trajectory that enters $v$.
For every directed edge $e = (u,v)$, we write $-e$ for the reversed edge, that is, $-e = (v,u)$.  The \defn{$W$-Aldous--Broder forest at time $t$} is
\begin{equation*}
\AB_W(\I_W(t)) = \{-e_t(v;W):v\in V\sm W\}.
\end{equation*}
For brevity, we write $\AB_W(t):=\AB_W(\I_W(t))$.
The concatenation of the $W$-trajectories starting from time $t$ has the distribution of the lazy random walk on $G/W$ starting from the contracted vertex of $W$. Hence, the correctness of Aldous--Broder algorithm (see in \cref{sec:ab}) implies that $\AB_W(t)$ has the distribution of $\cT_W$. 
Note that if $T\subseteq G[W]$ is a tree that spans $W$, then the union of that tree and $\AB_W(t)$ is a spanning tree of $G$ (when forgetting edge orientations).
The spatial Markov property of uniform spanning trees (\cref{cl:spatial}) implies that the union of $\AB_W(t)$ and $T$ is in fact the $\UST$ of $G$ conditioned to have $T$ as a subgraph.

Recall that for an oriented rooted forest $T$ and a vertex $v$ we denote by $\past^T(v)$ the \defn{past} of $v$ in $T$ and by $\fut^T(v)$ the \defn{future} of $v$ in $T$. 
For $t\in\RR$ and $v\in V$ set $\past^W_t(v)=\past^{\AB_W(t)}(v)$ and similarly $\fut^W_t(v)=\fut^{\AB_W(t)}(v)$, and note that it follows that $\past^W_t(v)$ and and $\fut^W_t(v)$ have the same law as $\past^{\cT_W}(v)$ and $\fut^{\cT_W}(v)$, respectively.

For every $(\gamma, t) \in \cW_W \times \RR$ and $x\in\RR$ we set $\Phi_x(\gamma,t) = (\gamma,t+x)$. We extend $\Phi_x$ to subsets of $H \subseteq \cW_W \times \RR$ by
\begin{equation*}
  \Phi_x(H) = \left\{(\gamma,t+x) \midd (\gamma,t)\in H \right\}. 
\end{equation*}
Since the measure $\mathcal{L}$ is invariant with respect to the operation $t \to t+x$ for every $x\in\RR$, we obtain the following.
\begin{observation} \label{obs:time:shift}
	Let $W$ be a nonempty vertex set.  For every $t,x \in \RR$, the law of $\Phi_x(\I_W(t))$ is equal to the law of $\I_W(t+x)$.
\end{observation}
Denote
\begin{equation*}
  I_W[a,b] = \left\{v\in V \midd \sigma_a(v;W) \leq b\right\}.
\end{equation*}
That is, $I_W[a,b]$ is the set of vertices which have a directed edge emanating from them in a $W$-trajectory of $\I_W[a,b]$.

\begin{claim}\label{cl:chron}
  Let $W$ be a nonempty vertex set.
	For every $u,v\in V$ and $t\in\RR$, if $v\in\past^{W}_t(u)$ then
	$\sigma_t(v;W)\ge \sigma_t(u;W)$.
\end{claim}

\begin{proof}
  First note that if $e=(w_0,w_1)$ is a directed edge in $\AB_W(t)$ then by definition
  $\sigma_t(w_0;W)\ge\sigma_t(w_1;W)$, since the first trajectory edge entering $w_0$ is $(w_1,w_0)$.
  Extending this argument to simple paths, if $(w_0,\ldots,w_k)$ is a directed path in $\AB_W(t)$ then for $0\le i<k$ we have that $\sigma_t(w_i;W)\ge\sigma_t(w_{i+1};W)$.
  The claim follows since if $v\in\past^W_t(u)$ then there exists a directed path from $v$ to $u$ in $\AB_W(t)$.
\end{proof}

\begin{claim}\label{cl:pasts}	
	Let $W$ be a nonempty vertex set.
  For every $u\in V$ and $a\le b$, if $u\notin I_W[a,b]$ then $\past^{W}_a(u) \subseteq \past^{W}_b(u)$.
\end{claim}

\begin{proof}
	Let $v\in \past^{W}_a(u)$.
  We show that $v \in \past^{W}_b(u)$ and that the unique directed path from $v$ to $u$ in $\past^{W}_a(u)$ is also in $\past^{W}_b(u)$.  Since $u \notin I_W[a,b]$, we have that $\sigma_a(u;W)\ge b$. Thus, by \cref{cl:chron} the directed path $\gamma = (v=w_0,w_1,\ldots,w_k = u)$ in $\AB_W(a)$ has $\sigma_a(w_i;W) \geq b$ for every $0\leq i \leq k$. Therefore, for every $0 \leq i \leq k$ we have $\sigma_a(w_i;W) = \sigma_b(w_i;W)$. Thus, the first edge emanating from $w_i$ after time $a$ equals to the first edge emanating from $w_i$ after time $b$. Therefore, every edge in this path is also in $\past^{W}_b(u)$. Hence, $\gamma \subseteq \past^{W}_b(u)$, as required.
\end{proof}

The following stochastic domination lemma, which appears in a slightly different setting in~\cite{Hut18+}, will be useful. It essentially says that the past of any vertex $w$ can only grow when we wire this vertex to $W$.

\begin{lemma}\label{lem:stoc:dom}
  For any increasing event $A$ and every three vertices $u,v,w \in V$ we have that 
  \begin{equation*}
    \pr\left(\past^{\cT_{u,v}}(w) \in A \midd \fut^{\cT_{u,v}}(w)\right)
      \le \pr\left(\past^{\cT_{u,w}}(w) \in A\right).
  \end{equation*}
\end{lemma}

\begin{proof}
	Let $\varphi$ be a simple path in $G$ from $w$ to $\{u,v\}$.
  Applying \cref{cl:spatial,lem:neg:ass} we have that $\past^{\cT_{u,v}}(w)$ conditioned on $\fut^{\cT_{u,v}}(w)=\varphi$ is stochastically dominated by $\past^{\cT_{u,v,w}}(w)$.
  By applying \cref{lem:neg:ass} again we have that $\past^{\cT_{u,v,w}}(w)$ is stochastically dominated by $\past^{\cT_{u,w}}(w)$.
\end{proof}

\subsection{The height of the past} \label{sec:recursiveinequality}
For an oriented rooted forest $T$ denote by $\height(T)$ its \defn{height}, namely, the length of the longest directed simple path in $T$.
Let $W$ be a nonempty vertex set and let $u$ be a vertex.
Recall that $\past^W(u)$ denotes the past of $u$ in $\cT_W$, that is, the oriented subtree of $\cT_W$ in which all paths are directed towards $u$.
For a vertex set $W$ and a vertex $u$, denote $W_u=W\cup\{u\}$.
For an integer $\ell\ge 1$ define
\begin{equation*}
  \cQ_W(\ell) = \max_{u\in V}
    \pr\left(\height\left(\past^{W_u}(u)\right)\ge\ell\right).
\end{equation*}

\begin{claim}\label{cl:q:ind}
  Let $W$ be a nonempty vertex set.
  Then, for every integer $\ell\ge 1$,
  \begin{equation*}
    \cQ_W(\ell) \le \frac{C_2}{\ell} + \frac{\cQ_W(\ell/3)}{4},
  \end{equation*}
  for $C_2 = 17280\cdot D^4\bub_W(G)^3\log(192D\bub_W(G))$.
\end{claim}

\begin{proof}
  Fix an integer $\ell\ge 1$.
	Throughout the proof, we will write $\sigma_t(u)$ instead of $\sigma_t(u;W)$.
	Let $u\in V$ (possibly in $W$) and consider the $W_u$-wired random interlacement process.  
  For $v\in V$ and $t\in\RR$ let $\sB^\ell_t(v,u)$ be the event that
	$v \in \past_t^{W_u}(u)$ and that the length of the unique directed simple path from $v$ to $u$ in $\past^{W_u}_t(u)$ is in $[\ell/3,2\ell/3]$.
	Let $\sC^\ell_t(v,u)$ be the event 
	that $\sB^\ell_t(v,u)$ occurs and that there is a simple path of length $\ell$ ending at $u$ and passing through $v$ in $\past_t^{W_u}(u)$.
	Fix $\zeta>0$ (to be chosen later), let $Z$ count the number of vertices $v\in V$ for which $\sC^\ell_0(v,u)$ occurs and let $Z_\zeta = Z\cdot\ind_{\sigma_0(u)>\zeta}$.  By 
	Markov's inequality,
  \begin{equation}\label{eq:qind:main}
  \begin{aligned}
    \pr\left(\height\left(\past^{W_u}(u)\right)\ge\ell\right)
    &\le \pr(\sigma_0(u) \leq \zeta) + \pr(Z_\zeta\ge \ell/3) \\
    &\le \pr(\sigma_0(u) \leq \zeta)
       + \frac{3}{\ell} \cdot \sum_{v\in V}
         \pr\left(\sC^\ell_0(v,u) \pand \{\sigma_0(u)>\zeta\}\right).
	\end{aligned}  
  \end{equation}
	It is very simple to control $\pr(\sigma_0(u) \leq \zeta)$.
  Write $\delta=\delta(G)$ and $\Delta=\Delta(G)$.
  Note that the number of $W_u$-trajectories that emanate from $u$ in a time interval of length $\zeta$ has a Poisson distribution with mean $\zeta\cdot \d(u)/\Vol(W_u)$.
  Hence, by Markov's inequality
  \begin{equation}\label{eq:qind:sigma}
    \pr(\sigma_0(u)\leq \zeta) \le \frac{\zeta \d(u)}{\Vol(W_u)}
    \le \frac{\zeta\Delta}{\delta|W_u|} \le\frac{\zeta D}{|W|}.
  \end{equation}
	As for the second term in \eqref{eq:qind:main}, denote $\sF^\ell(v,u) = \left\{\sC^\ell_0(v,u) \pand \sigma_0(u)>\zeta\right\}$.
  For $t\in\RR$, let $\phi_t(v,u)$ be the unique directed path from $v$ to $u$ in $\past^{W_u}_t(u)$, if such exists.
	Note that if $\phi_t(v,u) = (v=v_0,v_1,\ldots,v_k=u)$, by \cref{cl:chron} for every $1\leq i \leq k$ we have that $\sigma_t(v_{i-1})\ge\sigma_t(v_i)$.  
  Therefore, $\sigma_0(u)> \zeta$ implies $\phi_0(v,u)\cap I_{W_u}[0,\zeta]=\es$, and 
	clearly the reverse implication also holds.
  Hence,
  \begin{equation}\label{eq:qind:e:c:1}
    \pr\left(\sF^\ell(v,u)\right) 
    = \pr\left(\sC^\ell_0(v,u)
               \pand \phi_0(v,u)\cap I_{W_u}[0,\zeta] = \es\right).
  \end{equation}
	We now claim that the following inequality holds.
  \begin{equation*}\label{eq:qind:c}
  \begin{aligned}
    &\pr\left(\sC^\ell_0(v,u)
             \pand \phi_0(v,u)\cap I_{W_u}[0,\zeta] = \es\right)
    \le \pr\left(\sC^\ell_\zeta(v,u)
                 \pand \phi_\zeta(v,u)\cap I_{W_u}[0,\zeta] = \es\right).
  \end{aligned}
  \end{equation*}
	Indeed, assume the event $\{\sC^\ell_0(v,u) \pand \phi_0(v,u)\cap I_{W_u}[0,\zeta] = \es\}$ holds. Then,  in $\AB_{W_u}(0)$
	there exists a directed path of length $\ell$ passing through $v$ and ending in $u$.
	Let $w$ be the first vertex in this path.
  Since $\sigma_0(u)>\zeta$, we have by \cref{cl:pasts} that $\past^{W_u}_0(u)\subseteq\past^{W_u}_\zeta(u)$ and $\phi_0(w,u)=\phi_\zeta(w,u)$.
  Thus $\{\sC^\ell_\zeta(v,u) \pand \phi_\zeta(v,u)\cap I_{W_u}[0,\zeta] = \es \}$ holds and the inequality follows.
  Therefore by \eqref{eq:qind:e:c:1} and \cref{obs:time:shift} we obtain
  \begin{equation*}\label{eq:qind:e:c:2}
    \pr\left(\sF^\ell(v,u)\right)
      \leq \pr\left(\sC^\ell_0(v,u)
                    \pand \phi_0(v,u)\cap I_{W_u}[-\zeta,0] = \es\right).
	\end{equation*}
  The events $\sC^\ell_0(v,u)$ and $\{ \phi_0(v,u)\cap I_{W_u}[-\zeta,0] = \es \}$ are conditionally independent conditional on the event $\sB^\ell_0(v,u)$ and the random path $\phi_0(v,u)$, since the first event depends only on trajectories emanating at positive times, while the second depends only on trajectories emanating at negative times.
  Thus, we have that almost surely
	\begin{equation}\label{eq:q:ind:cond:ind}
  \begin{aligned}
    &\pr\left(\sC^\ell_0(v,u)
	            \pand \phi_0(v,u)\cap I_{W_u}[-\zeta,0] = \es 
               \midd \sB^\ell_0(v,u),\ \phi_0(v,u)\right)\\
    &=\pr\left(\sC^\ell_0(v,u)\midd \sB^\ell_0(v,u),\ \phi_0(v,u)\right)
      \cdot \pr\left(\phi_0(v,u)\cap I_{W_u}[-\zeta,0] = \es
                     \midd \sB^\ell_0(v,u),\ \phi_0(v,u)\right).
  \end{aligned}
  \end{equation}
	We will now bound the first of the two terms of the right hand side of~\eqref{eq:q:ind:cond:ind}. For the event $\sC^\ell_0(v,u)$ to hold,
  the past of $v$ in $\past^{W_u}_0(u)$ must be of height at least $\ell/3$,
  thus
	\begin{equation*}
	  \pr\left(\sC^\ell_0(v,u)\midd \sB^\ell_0(v,u),\ \phi_0(v,u)\right)
    \le \pr\left(\height\left(\past^{W_u}_0(v)\right)\ge\ell/3
                 \midd \sB^\ell_0(v,u),\ \phi_0(v,u)\right).
	\end{equation*}
  Note that on the event $\sB^\ell_0(v,u)$ we have $\phi_0(v,u)=\fut^{W_u}_0(v)$.
  Thus, by \cref{lem:stoc:dom},
  \begin{equation*}
    \pr\left(\height\left(\past^{W_u}_0(v)\right)\ge\ell/3
             \midd \sB^\ell_0(v,u),\ \phi_0(v,u)\right)
    \le\pr\left(\height\left(\past^{W_v}(v)\right)\ge\ell/3 \right) \le \cQ_W(\ell/3).
  \end{equation*}
	Therefore, averaging over $\phi_0(v,u)$,
	\begin{equation}\label{eq:qind:e:q}
    \pr\left(\sF^\ell(v,u)\right)
      \le \cQ_W(\ell/3) \cdot
        \pr\left(\phi_0(v,u)\cap I_{W_u}[-\zeta,0] = \es
                 \midd \sB^\ell_0(v,u)\right)
        \cdot \pr\left(\sB^\ell_0(v,u)\right).
	\end{equation}
	Note that if a path $\phi$ is ``close'' to $W_u$, in the sense that $\ceff(W_u\lr\phi)$ is large, then there is a high probability that a trajectory emanating from $W_u$ will hit $\phi$.
  We would like to distinguish between paths which are ``close'' to $W_u$ from those which are ``far'' from it.
  For $t\in\RR$, let $\hat\phi_t(v,u)$ be the unique directed path from $v$ to $u$ in $\past^{W_u}_t(u)$, if such exists, with $u$ excluded, namely, $\hat\phi_t(v,u)=\phi_t(v,u)\sm\{u\}$.
  Since the number of trajectories emanating from $W_u$
   from time $-\zeta$ to time $0$ which hit $\hat\phi_0(v,u)$ is a Poisson random variable with mean $\zeta\cdot\ceff(W_u\lr\hat\phi_0(v,u))/\Vol(W_u)$,
  \begin{equation*}
    \pr\left(\phi_0(v,u)\cap I_{W_u}[-\zeta,0] = \es
             \pand \ceff(W_u\lr\hat\phi_0(v,u)) \geq \xi\ell\Delta
             \midd \sB^\ell_0(v,u)\right)
    \le \exp\left(-\frac{\zeta\xi\ell\Delta}{\Vol(W_u)}\right),
  \end{equation*}
  for any $\xi>0$.
  On the other hand,
  \begin{align*}
    &\pr\left(\phi_0(v,u)\cap I_{W_u}[-\zeta,0] = \es
              \pand \ceff(W_u \lr \hat\phi_0(v,u)) \leq \xi\ell\Delta
              \midd \sB^\ell_0(v,u)\right)
        \cdot\pr\left(\sB^\ell_0(v,u)\right)\\
  	&\le\pr\left(\sB^\ell_0(v,u)
                 \pand \ceff(W_u \lr \hat\phi_0(v,u)) \leq \xi\ell\Delta\right).
  \end{align*}
  By \cref{lem:bad:bd},
	\begin{equation*}
	  \sum_{v \in V}
      \pr\left(\sB^\ell_0(v,u)
               \pand \ceff(W_u \lr \hat\phi_0(v,u)) \leq \xi\ell\Delta\right)
    \le C_1\xi\ell,
	\end{equation*}
  for $C_1=360D^3\bub_W(G)^3$.
  We then have, using~\eqref{eq:qind:e:q},
  \begin{equation*}\label{eq:qind:sum}
  \begin{aligned}
        &\frac{3}{\ell}\sum_{v \in V}\pr\left(\sF^\ell(v,u)\right)\\
    &\le \frac{3}{\ell}\cQ_W\left(\frac{\ell}{3}\right)\sum_{v \in V}
         \pr\left(\phi_0(v,u)\cap I_{W_u}[-\zeta,0] = \es
                  \midd \sB^\ell_0(v,u)\right)
         \cdot \pr\left(\sB^\ell_0(v,u)\right)\\
    &\le \frac{3}{\ell}\cQ_W\left(\frac{\ell}{3}\right)
         \left(
           C_1\xi\ell
           + \exp\left(-\frac{\zeta\xi\ell\Delta}{\Vol(W_u)}\right)
             \cdot \E\left[\left|\ball_{\cT_{W_u}}(0,\ell)\right|\right]
         \right).
  \end{aligned}
	\end{equation*}
	We then choose $\zeta=K|W|/\ell$ for some $K>0$ that will be chosen later and use \cref{cl:hang:ball} to bound the above by
	\begin{equation}\label{eq:qind:sum:end}
    3\cQ_W\left(\frac{\ell}{3}\right)
      \left(C_1\xi + 8D\bub_W(G)\exp\left(-K\xi/2\right)\right),
	\end{equation}
  where we used $2|W|\Delta\ge\Vol(W_u)$.
	We choose $\xi = (24C_1)^{-1}$ 
  and 
    $K=2\xi^{-1}\log(192D\bub_W(G))$,
  so that~\eqref{eq:qind:sum:end} is upper bounded by $\cQ_W(\ell/3)/4$.
	To conclude, by~\eqref{eq:qind:main}, \eqref{eq:qind:sigma} and \eqref{eq:qind:sum:end} we have
  \begin{equation*}
    \pr\left(\height\left(\past^{W_u}(u)\right)\ge\ell\right)	
    \le \frac{\zeta D}{|W|} + \frac{1}{4}\cQ_W\left(\frac{\ell}{3}\right)
    = \frac{KD}{\ell} + \frac{1}{4}\cQ_W\left(\frac{\ell}{3}\right),
  \end{equation*}
  so the claim holds for $C_2 = KD=17280\cdot D^4\bub_W(G)^3\log(192D\bub_W(G))$.
\end{proof}

\begin{corollary} \label{cor:ub:ind}
  Let $W$ be a nonempty vertex set.
  Then, for every integer $\ell\ge 1$,
  \begin{equation*}
    \cQ_W(\ell) \leq \frac{C_3}{\ell},
  \end{equation*}
  where $C_3 = 4C_2 = 69120\cdot D^4\bub_W(G)^3\log(192D\bub_W(G))$ for $C_2$ from \cref{cl:q:ind}.
\end{corollary}

\begin{proof}
  Using \cref{cl:q:ind} and by induction,
  \begin{equation*}
    \cQ_W(\ell)
    \le \frac{C_2}{\ell} + \frac{\cQ_W(\ell/3)}{4}
    \le \frac{C_2}{\ell} + \frac{3}{4}\cdot\frac{C_3}{\ell}
    \le \frac{C_3}{\ell}. \qedhere
  \end{equation*}
\end{proof}

\begin{lemma}\label{lem:mns}
  Let $W$ be a nonempty vertex set.
  Then, for every integer $\ell\ge 1$,
	\begin{equation*}
    \pr(\height(\cT_W) \geq \ell) \leq \frac{C_3|W|}{\ell},
	\end{equation*}	
\end{lemma}
	for $C_3 = 69120\cdot D^4\bub_W(G)^3\log(192D\bub_W(G))$.
\begin{proof}
	Follows directly by the union bound and \cref{cor:ub:ind}.
\end{proof}

\subsection{Finishing up the proof}
In view of \cref{lem:mns}, it suffices to bound the $W$-bubble sum of a set $W$ with large enough capacity in order to prove \cref{thm:ub}.
This is the goal of the next simple claim.
Recall that $G=(V,E)$ be a connected graph on $n$ vertices, satisfying \ref{A:bal}, \ref{A:mix} and \ref{A:esc} with parameters $D,\alpha,\theta$ respectively, and that $\Cap_\r(\cdot)$ denotes the $\r$-capacity as defined in~\eqref{eq:cap}, where $\r=\r(\alpha)$ is as defined in~\eqref{eq:rs}.

\begin{claim}\label{cl:bubw}
  Let $\chi>0$ and let $W$ be a nonempty vertex set with $\Cap_\r(W)\ge \chi \r/\sqrt{n}$.  Then
  \begin{equation*}
    \bub_W(G)\le \theta + 2D + \frac{36D}{\chi^2}.
  \end{equation*}
\end{claim}

\begin{proof}
  Let $v$ be a vertex, and let $X$ be a lazy random walk on $G$.
  For $t\ge\s$ let $F_W(t)$ be the event that $W\cap X[0,t-\s)=\es$.
  By \cref{cl:cap:mix},
  \begin{equation*}
    \pr_v\left(F_W(t)\right)
    \le \left( 1 - \frac{\chi\r}{3\sqrt{n}}\right)^{\floor{(t-\s)/\r}}.
  \end{equation*}
  Thus, using \cref{fact:mix},
  \begin{equation*}
  	\pvec{W}^t(v,v)
      \leq \pr_v\left(F_W(t)\right)
           \cdot \pr_v\left(X_t=v\midd F_W(t)\right)
      \leq \left(1-\frac{\chi \r}{3\sqrt{n}}\right)^{\floor{(t-\s)/\r}} 
        \cdot\frac{2D}{n}.
  \end{equation*}
  We then have, using \cref{fact:esc} and the equality $\sum_{t=0}^\infty(t+1)x^t = (1-x)^{-2}$,
  \begin{align*}
  	\bub_W(G)
      = \sum_{t=0}^{\infty}(t+1)\sup_v\pvec{W}^t(v,v) 
  	  &\leq \sum_{t=0}^{\s}(t+1)\sup_v\vect{p}^t(v,v)
           +\sum_{k=0}^{\infty}
              \sum_{t=k\r+\s}^{(k+1)\r+s-1}(t+1)\sup_v\pvec{W}^t(v,v) \\
      &\leq \theta + 2D
           +\sum_{k=0}^{\infty}
              \sum_{t=k\r+\s}^{(k+1)\r+\s-1}2(k+1)\r
                \left(1-\frac{\chi \r}{3\sqrt{n}}\right)^{\floor{(t-\s)/\r}} \cdot\frac{2D}{n}\\
  	  &\leq \theta + 2D + \frac{4r^2D}{n}\sum_{k=0}^{\infty}
  	(k+1)\left(1-\frac{\chi \r}{3\sqrt{n}}\right)^k 
  	\leq \theta + 2D + \frac{36D}{\chi^2}.\qedhere
  \end{align*}
\end{proof}

\begin{proof}[Proof of \cref{thm:ub}]
  By \cref{lem:mns},
  \begin{equation*}
    \pr(\diam(\UST(G/W)) \ge C\sqrt{n})
    \le \pr(\height(\cT^W) \ge C\sqrt{n}/2)
    \le 2C_3 A/C.
  \end{equation*}
  By \cref{cl:bubw}, $C_3=C_3(D,\alpha,\theta,\chi)$ is bounded.
  The result follows by taking $C\ge 2C_3A/\eps$.
\end{proof}

\begin{proof}[Proof of the upper bound in \cref{thm:main}]
  Let $\eps>0$.
  Let $u,v$ be two independent stationary points in $G$ and let $\phi$ be the unique simple path between them in $\UST(G)$.
  By \cref{thm:lb} there exist $\chi,A>0$ depending only on $D,\alpha,\theta,\eps$ such that with probability at least $1-\eps/2$ we have that $|\phi|\le A\sqrt{n}$ and $\Cap_r(\phi)\ge \chi\r n^{-1/2}$.
  By \cref{cl:spatial,thm:ub}, on that event there exists $C=C(D,\alpha,\theta,\eps)$ such that with probability at least $1-\eps/2$,
  \begin{equation*}
    \diam(\UST(G)) \le |\phi| + \diam(\UST(G/\phi)) \le (A+C)\sqrt{n}. \qedhere
  \end{equation*}
\end{proof}

\begin{acknowledgements}
The authors wish to thank Tom Hutchcroft and Yinon Spinka for useful discussions and for their helpful comments on the paper.
This research is supported by ERC starting grant 676970 RANDGEOM and by ISF grant 1207/15.
\end{acknowledgements}

\bibliography{library}

\begin{bibdiv}
\begin{biblist}

\bib{Ald90}{article}{
      author={Aldous, David~J.},
       title={The random walk construction of uniform spanning trees and
  uniform labelled trees},
        date={1990},
        ISSN={0895-4801},
     journal={SIAM Journal on Discrete Mathematics},
      volume={3},
      number={4},
       pages={450\ndash 465},
         url={https://doi.org/10.1137/0403039},
      review={\MR{1069105}},
}

\bib{ACRT1}{article}{
      author={Aldous, David~J.},
       title={The continuum random tree. {I}},
        date={1991},
        ISSN={0091-1798},
     journal={The Annals of Probability},
      volume={19},
      number={1},
       pages={1\ndash 28},
  url={http://links.jstor.org/sici?sici=0091-1798(199101)19:1<1:TCRTI>2.0.CO;2-B&origin=MSN},
      review={\MR{1085326}},
}

\bib{ACRT2}{incollection}{
      author={Aldous, David~J.},
       title={The continuum random tree. {II}. {A}n overview},
        date={1991},
   booktitle={Stochastic analysis ({D}urham, 1990)},
      series={London Math. Soc. Lecture Note Ser.},
      volume={167},
   publisher={Cambridge Univ. Press, Cambridge},
       pages={23\ndash 70},
         url={https://doi.org/10.1017/CBO9780511662980.003},
      review={\MR{1166406}},
}

\bib{ACRT3}{article}{
      author={Aldous, David~J.},
       title={The continuum random tree. {III}},
        date={1993},
        ISSN={0091-1798},
     journal={The Annals of Probability},
      volume={21},
      number={1},
       pages={248\ndash 289},
  url={http://links.jstor.org/sici?sici=0091-1798(199301)21:1<248:TCRTI>2.0.CO;2-1&origin=MSN},
      review={\MR{1207226}},
}

\bib{BLPS01}{article}{
      author={Benjamini, Itai},
      author={Lyons, Russell},
      author={Peres, Yuval},
      author={Schramm, Oded},
       title={Uniform spanning forests},
        date={2001},
        ISSN={0091-1798},
     journal={The Annals of Probability},
      volume={29},
      number={1},
       pages={1\ndash 65},
         url={https://doi.org/10.1214/aop/1008956321},
      review={\MR{1825141}},
}

\bib{Bro89}{inproceedings}{
      author={Broder, Andrei},
       title={Generating random spanning trees},
        date={1989},
   booktitle={Proceedings of the 30th {A}nnual {S}ymposium on {F}oundations of
  {C}omputer {S}cience},
      series={SFCS '89},
   publisher={IEEE Computer Society},
     address={Washington, DC, USA},
       pages={442\ndash 447},
         url={https://doi.org/10.1109/SFCS.1989.63516},
}

\bib{CHL12}{article}{
      author={Chung, Fan},
      author={Horn, Paul},
      author={Lu, Linyuan},
       title={Diameter of random spanning trees in a given graph},
        date={2012},
        ISSN={0364-9024},
     journal={Journal of Graph Theory},
      volume={69},
      number={3},
       pages={223\ndash 240},
         url={https://doi.org/10.1002/jgt.20577},
      review={\MR{2898865}},
}

\bib{DSC94}{article}{
      author={Diaconis, Persi~W.},
      author={Saloff-Coste, Laurent},
       title={Moderate growth and random walk on finite groups},
        date={1994},
        ISSN={1016-443X},
     journal={Geometric and Functional Analysis},
      volume={4},
      number={1},
       pages={1\ndash 36},
         url={https://doi.org/10.1007/BF01898359},
      review={\MR{1254308}},
}

\bib{FM92}{inproceedings}{
      author={Feder, Tom\'{a}s},
      author={Mihail, Milena},
       title={Balanced matroids},
        date={1992},
   booktitle={Proceedings of the {T}wenty-fourth {A}nnual {ACM} {S}ymposium on
  {T}heory of {C}omputing},
      series={STOC '92},
   publisher={ACM},
     address={New York, NY, USA},
       pages={26\ndash 38},
         url={http://doi.acm.org/10.1145/129712.129716},
}

\bib{Hoe63}{article}{
      author={Hoeffding, Wassily},
       title={Probability inequalities for sums of bounded random variables},
        date={1963},
        ISSN={0162-1459},
     journal={Journal of the American Statistical Association},
      volume={58},
       pages={13\ndash 30},
  url={http://links.jstor.org/sici?sici=0162-1459(196303)58:301<13:PIFSOB>2.0.CO;2-D&origin=MSN},
      review={\MR{0144363}},
}

\bib{Hut18}{article}{
      author={Hutchcroft, Tom},
       title={Interlacements and the wired uniform spanning forest},
        date={2018},
        ISSN={0091-1798},
     journal={The Annals of Probability},
      volume={46},
      number={2},
       pages={1170\ndash 1200},
         url={https://doi.org/10.1214/17-AOP1203},
      review={\MR{3773383}},
}

\bib{Hut18+}{article}{
      author={{Hutchcroft}, Tom},
       title={{Universality of high-dimensional spanning forests and
  sandpiles}},
        date={2018Apr},
     journal={arXiv e-prints},
      eprint={1804.04120},
        note={To appear in ``Probability Theory and Related Fields''},
}

\bib{KN09}{article}{
      author={Kozma, Gady},
      author={Nachmias, Asaf},
       title={The {A}lexander-{O}rbach conjecture holds in high dimensions},
        date={2009},
        ISSN={0020-9910},
     journal={Inventiones Mathematicae},
      volume={178},
      number={3},
       pages={635\ndash 654},
         url={https://doi.org/10.1007/s00222-009-0208-4},
      review={\MR{2551766}},
}

\bib{Law80}{article}{
      author={Lawler, Gregory~F.},
       title={A self-avoiding random walk},
        date={1980},
        ISSN={0012-7094},
     journal={Duke Mathematical Journal},
      volume={47},
      number={3},
       pages={655\ndash 693},
         url={http://projecteuclid.org/euclid.dmj/1077314188},
      review={\MR{587173}},
}

\bib{LPW2}{book}{
      author={Levin, David~A.},
      author={Peres, Yuval},
       title={Markov chains and mixing times},
   publisher={American Mathematical Society, Providence, RI},
        date={2017},
        ISBN={978-1-4704-2962-1},
        note={Second edition of [ MR2466937], With contributions by Elizabeth
  L. Wilmer, With a chapter on ``Coupling from the past'' by James G. Propp and
  David B. Wilson. Available at
  \url{https://pages.uoregon.edu/dlevin/MARKOV/mcmt2e.pdf}},
      review={\MR{3726904}},
}

\bib{LP}{book}{
      author={Lyons, Russell},
      author={Peres, Yuval},
       title={Probability on trees and networks},
      series={Cambridge Series in Statistical and Probabilistic Mathematics},
   publisher={Cambridge University Press, New York},
        date={2016},
      volume={42},
        ISBN={978-1-107-16015-6},
         url={https://doi.org/10.1017/9781316672815},
      review={\MR{3616205}},
}

\bib{Pem91}{article}{
      author={Pemantle, Robin},
       title={Choosing a spanning tree for the integer lattice uniformly},
        date={1991},
        ISSN={0091-1798},
     journal={The Annals of Probability},
      volume={19},
      number={4},
       pages={1559\ndash 1574},
  url={http://links.jstor.org/sici?sici=0091-1798(199110)19:4<1559:CASTFT>2.0.CO;2-E&origin=MSN},
      review={\MR{1127715}},
}

\bib{PR04+}{article}{
      author={{Peres}, Yuval},
      author={{Revelle}, David},
       title={{Scaling limits of the uniform spanning tree and loop-erased
  random walk on finite graphs}},
        date={2004Oct},
     journal={arXiv Mathematics e-prints},
      eprint={math/0410430},
}

\bib{Sch09}{article}{
      author={Schweinsberg, Jason},
       title={The loop-erased random walk and the uniform spanning tree on the
  four-dimensional discrete torus},
        date={2009},
        ISSN={0178-8051},
     journal={Probability Theory and Related Fields},
      volume={144},
      number={3-4},
       pages={319\ndash 370},
         url={https://doi.org/10.1007/s00440-008-0149-7},
      review={\MR{2496437}},
}

\bib{Sze83}{incollection}{
      author={Szekeres, George},
       title={Distribution of labelled trees by diameter},
        date={1983},
   booktitle={Combinatorial {M}athematics {X} ({A}delaide, 1982), {L}ecture
  {N}otes in {M}athematics},
      volume={1036},
   publisher={Springer, Berlin},
       pages={392\ndash 397},
         url={https://doi.org/10.1007/BFb0071532},
      review={\MR{731595}},
}

\bib{Szn10}{article}{
      author={Sznitman, Alain-Sol},
       title={Vacant set of random interlacements and percolation},
        date={2010},
        ISSN={0003-486X},
     journal={Annals of Mathematics. Second Series},
      volume={171},
      number={3},
       pages={2039\ndash 2087},
         url={https://doi.org/10.4007/annals.2010.171.2039},
      review={\MR{2680403}},
}

\bib{Wil96}{inproceedings}{
      author={Wilson, David~Bruce},
       title={Generating random spanning trees more quickly than the cover
  time},
        date={1996},
   booktitle={Proceedings of the {T}wenty-eighth {A}nnual {ACM} {S}ymposium on
  the {T}heory of {C}omputing ({P}hiladelphia, {PA}, 1996)},
   publisher={ACM, New York},
       pages={296\ndash 303},
         url={https://doi.org/10.1145/237814.237880},
      review={\MR{1427525}},
}

\end{biblist}
\end{bibdiv}

\end{document}